\documentclass[12pt,reqno]{amsart}
\usepackage{amsmath,amssymb,amscd,latexsym,amsthm,mathrsfs}
\usepackage{color}
\usepackage[unicode]{hyperref}
\usepackage{hypbmsec,longtable}
\ifx\pdfoutput\undefined\else
\hypersetup{
colorlinks=true,
linkcolor=blue,
citecolor=blue,
urlcolor=blue,
filecolor=blue,
bookmarksnumbered=true,
pdfstartview=FitH,
pdfhighlight=/N
}
\fi
\newtheorem{theorem}{Theorem}
\newtheorem{lemma}{Lemma}

\newtheorem{proposition}{Proposition}
\theoremstyle{remark}

\newtheorem*{acknowledgements}{Acknowledgements}

\renewcommand{\Im}{\operatorname{Im}}
\renewcommand{\Re}{\operatorname{Re}}

\newcommand\Li{\operatorname{Li}}

\newcommand{\hpg}[5]{{}_{#1}\mbox{\rm F}_{\!#2}\!
  \left(\left.{#3 \atop #4}\right| #5 \right) }
\newcommand{\ff}{\operatorname{F}}

\begin{document}

\hypersetup{pdfauthor={Jesus Guillera, Mathew Rogers},%
pdftitle={Ramanujan series upside-down}}

\title{Ramanujan series upside-down}
\author{Jesus Guillera}
\address{Av.\ Ces\'areo Alierta, 31 esc.~izda 4$^\circ$--A, Zaragoza, SPAIN}
\email{jguillera@gmail.com}

\author{Mathew Rogers}
\address{Department of Mathematics and Statistics, Universit\'e de Montr\'eal,
CP 6128 succ.\ Centre-ville, Montr\'eal Qu\'ebec H3C\,3J7, Canada}
\email{mathewrogers@gmail.com}

\date{June 15, 2012.}

\subjclass[2010]{Primary 33C20; Secondary 11F11, 11F03, 11Y60, 33C75, 33E05} \keywords{Dirichlet $L$-values, formulas for
$\pi$, hypergeometric series, lattice sums, Ramanujan}

\begin{abstract} We prove that there is a correspondence between Ramanujan-type formulas for $1/\pi$, and formulas for Dirichlet $L$-values.  If we have an identity of the form
\begin{equation*}
\frac{1}{\pi}=\sum_{n=0}^{\infty}\frac{(s)_n (\frac{1}{2})_n (1-s)_n}{n!^3}(a+b n )z^n,
\end{equation*}
where $(s)_n=\Gamma(s+n)/\Gamma(s)$,
then under certain conditions we prove that
\begin{equation*}
\sum_{n=1}^{\infty}\frac{n!^3}{(s)_n (\frac{1}{2})_n (1-s)_n}\frac{(a-b n )}{n^3}z^{-n}
\end{equation*}
reduces to Dirichlet $L$-values evaluated at $2$.  The two sums rarely converge at the same time, however divergent formulas make sense when they are interpreted as values of analytically continued hypergeometric functions.  The same method also allows us to resolve certain values of the Epstein zeta function in terms of rapidly converging hypergeometric functions.  The Epstein zeta functions were previously studied by Glasser and Zucker in \cite{GZ}.
\end{abstract}

\maketitle

\section{Introduction}
\label{s-intro}
Quantities such as $\pi^2$ and the Dirichlet $L$-values are fundamental constants which appear in many areas of mathematics and physics.  It is interesting to relate them to hypergeometric functions, which  are important because of their applications in number theory.  For instance, Ap\'{e}ry proved the irrationality of $\zeta(3)$ using a $_4{\ff}_3$ identity \cite{Chudnovsky}.  Ramanujan discovered many famous hypergeometric formulas for $1/\pi$.  The following example \cite{Ra}:
\begin{equation}\label{Ramanujan formula 0}
\frac{1}{\pi}=\sum_{n=0}^{\infty}\frac{(-1)^n}{2^{6n}}{2n\choose n}^3 \left(\frac{1}{2}+2n\right),
\end{equation}
is connected to class number problems, and to the theory of complex multiplication \cite{Borwein}, \cite{Chudnovsky}.
In this paper we describe identities which we are closely related to Ramanujan's formulas.  Our first example can be constructed by manipulating \eqref{Ramanujan formula 0}.  Let $(1/2+2n)\mapsto(1/2-2n)$, flip the rest of the summand ``upside-down", insert a factor of $1/n^3$, and perform the summation for $n\ge1$.  Then we obtain a \textit{companion series identity}:
\begin{equation}\label{upside-down series 0}
8L_{-4}(2)=\sum_{n=1}^{\infty}\frac{(-1)^n 2^{6n}}{n^3{2n\choose n}^3}\left(\frac{1}{2}-2n\right).
\end{equation}
As usual $L_{-4}(2)=1-\frac{1}{3^2}+\frac{1}{5^2}\dots$ is Catalan's constant, $L_{k}(s):=\sum_{n=1}^{\infty}\frac{\chi_{k}(n)}{n^s}$ denotes the general Dirichlet $L$-series, and $\chi_k(n)=\left(\frac{k}{n}\right)$ is the Jacobi symbol.  Based on this example, we might expect that the same procedure should transform each of Ramanujan's formulas into identities involving Dirichlet $L$-values.  We prove that this guess is correct when certain technical conditions are added.  It is important to note that at least nine similar formulas already exist in the literature.  The individual formulas were discovered piecemeal with computational techniques, and proved by diverse methods.  We mention proofs due to Zeilberger \cite{ZB}, Guillera \cite{GuilleraThesis} \cite{G10Rama} \cite{G}, and the Hessami-Pilehroods \cite{HP}.  Sun also observed several identities from numerical experiments \cite{Sun}.  We give unified proofs of all of these results and conjectures in Theorem \ref{Theorem: proving Sun's formulas}.  We also show how to construct vast numbers of irrational formulas (such as \eqref{1/2 irrational pi^2 formula} and the examples in Table \ref{Table: Irrational evaluations}), which were previously unknown.  We describe our results in greater detail below.

Ramanujan identified
seventeen formulas for $1/\pi$ \cite{Ra}.  His identities all have the following form:
\begin{equation}\label{Ramanujan type series}
\frac{1}{\pi}=\sum_{n=0}^{\infty}\frac{(s)_n
\left(\frac{1}{2}\right)_n (1-s)_n}{(1)_n^3}(a+b n) z^n,
\end{equation}
where $(x)_n=\Gamma(x+n)/\Gamma(x)$.  Each example has
$s\in\{\frac{1}{2},\frac{1}{3},\frac{1}{4},\frac{1}{6}\}$,
with $(a, b, z)$ being parameterized by modular
functions \cite{Borwein},
\cite{Chudnovsky}.  When $s=\frac16$, $z=\frac{1}{j(\tau)}$, where $j(\tau)$ is the $j$-invariant, and the expressions for $a$ and $b$ involve Eisenstein series.  If we preserve the modular parameterizations for $(a, b, z)$, then
the general \textit{companion series} is given by
\begin{equation}\label{Comp series}
\sum_{n=1}^{\infty}\frac{(1)_n^3}{(s)_n(\frac{1}{2})_n(1-s)_n}\frac{\left(a-b
n\right)}{n^3}z^{-n}.
\end{equation}
When $n$ is large, standard asymptotics show that
{\allowdisplaybreaks $\frac{(s)_n\left(\frac{1}{2}\right)_n(1-s)_n}{(1)_n^3}\approx \frac{\sin(\pi s)}{(\pi n)^{3/2}}$}.
It follows that \eqref{Ramanujan type series} and \eqref{Comp series} can only converge simultaneously if $|z|=1$ (notice that \eqref{Ramanujan formula 0} and \eqref{upside-down series 0} occur when $s=\frac12$ and $(a,b,z)=\left(\frac{1}{2},2,-1\right)$).
Divergent cases still make sense, provided that each divergent infinite series is replaced by an analytically-continued hypergeometric function.  Once of the main goals of this work, is to transform divergent formulas for $1/\pi$, into interesting convergent formulas for Dirichlet $L$-values.

Suppose that $s\in\{\frac12,\frac13,\frac14\}$.  Then Propositions \ref{Proposition:Companion series reduction} and \ref{proposition: F in terms of S} reduce many values of the companion series \eqref{Comp series}, to linear combinations of two Epstein zeta function and elementary constants.  In general, once we fix the modular parameterizations for $(a,b,z)$ in \eqref{Comp series}, then Propositions \ref{Proposition:Companion series reduction} and \ref{proposition: F in terms of S} harshly restrict the domain of the modular functions (see the constraints on equations \eqref{final F(q) evaluation} and \eqref{final F(q) integer case}).  This means there are fewer \textit{potential} companion series evaluations, compared to the number of possible Ramanujan-type formulas coming from \eqref{Ramanujan type series}.  Finally, if the linear combination of Epstein zeta functions reduce to Dirichlet $L$-values, which is by no means automatic, then the companion series also reduces to Dirichlet $L$-values.  Proofs are based upon a new idea called \textit{completing the hypergeometric function}, which we outline in Section \ref{sec:completing the hypergeometric}.  The approach fails completely when $s=\frac16$, and we describe the rationale for this failure at the end of Section \ref{sec:completing the hypergeometric}.
The Epstein zeta functions which appear have been studied by Glasser and Zucker \cite{GZ}.  Following their notation, define
\begin{equation}\label{S-def}
S(A,B,C;t):=\sum_{(n,m)\ne (0,0)}\frac{1}{(A n^2+B n m+ C m^2)^t}.
\end{equation}
We demonstrate a calculation by proving \eqref{upside-down series 0}.  Set $q=-e^{-\pi\sqrt{2}}$ in \eqref{5F4 1/2 case in terms of F}. Then $(a,b,z)=(\frac{1}{2},2,-1)$.  By equation \eqref{final F(q) evaluation}, we have
\begin{equation*}
\sum_{n=1}^{\infty}\frac{(-1)^n}{n^3}\frac{(1)_n^3}{\left(\frac{1}{2}\right)_n^3}\left(\frac{1}{2}-2n\right)
=\frac{32\sqrt{2}}{\pi^2}\left(S(1,0,8;2)-S(3,4,4;2)\right).
\end{equation*}
Notice that $S(3,4,4;t)$ does not correspond to a reduced quadratic form ($C\ge A\ge|B|$),  but it is possible to show that $S(3,4,4;t)=S(3,2,3;t)$.
The key to completing the proof, is to reduce $S(A,B,C;t)$ to Dirichlet $L$-values.  It is fortunate that
this is a well-known problem.  Let us briefly recall that quadratic forms with fixed discriminant $D=B^2-4 A C$, are partitioned into equivalence classes under the action of $SL_2(\mathbb Z)$.  We say that quadratic forms of discriminant $D<0$ have \textit{one class per genus}, when disjoint classes of forms always represent disjoint sets of integers.   Glasser and Zucker conjectured that $S(A,B,C;t)$ reduces to Dirichlet $L$-values, if and only if $An^2+B n m+C m^2$ lives in a class of quadratic forms with one class per genus.  Despite the fact that Zucker and Robertson discovered a few strange counterexamples to this conjecture \cite{ZR}, most evidence suggests that the original conjecture is ``basically" correct.  Every interesting companion series boils down to two values of $S(A,B,C;2)$, and elementary constants.  The proof of \eqref{upside-down series 0} follows from showing
\begin{align*}
S(1,0,8;2)=&\frac{7\pi^2}{48} L_{-8}(2)+\frac{\pi
^2}{8\sqrt{2}}L_{-4}(2),\\
S(3,4,4;2)=&\frac{7\pi^2}{48} L_{-8}(2)-\frac{\pi
^2}{8\sqrt{2}}L_{-4}(2).
\end{align*}
This type of reasoning explains all of the previously known companion series formulas, and all of the results in Theorems \ref{Theorem: proving Sun's formulas} and \ref{Theorem: Divergent rational cases}.

There are many instances where it is probably impossible to express $S(A,B,C;t)$ in terms of Dirichlet $L$-values.  Then our method produces non-trivial hypergeometric formulas for $S(A,B,C;2)$.  For example, set $q=-e^{-\pi/3}$ in \eqref{5F4 1/2 case in terms of F}.  After some work we obtain
\begin{equation}\label{S(1,0,36) intro formula}
\frac{48}{\pi^2}S(1,0,36;2)=\frac{140}{27}L_{-4}(2)+\frac{13}{\sqrt{3}}L_{-3}(2)-\sum_{n=1}^{\infty}\frac{(1)_n^3}{(\frac12)_n^3}\frac{(a-b n)}{n^3}z^{-n},
\end{equation}
where
\begin{align*}
z=&-8\left(74977+40284r+21644 r^2+11629r^3\right),\\
a=&\frac{1}{18}\left(1038+558r+300r^2+161r^3\right),\\
b=&\frac{1}{3}\left(387+208r+112r^2+60r^3\right),
\end{align*}
and $r=\sqrt[4]{12}$.  Formula \eqref{S(1,0,36) intro formula} converges very rapidly because $z\approx -2.4\times 10^6$.  The infinite series can either be expressed as a ${_5{\ff}_4}$ function, or as a linear combination of two ${_4{\ff}_3}$'s.  In either case, this partially resolves a question of Zucker\footnote{Zucker's dream is to resolve $S(1,0,36;t)$ in terms of Dirichlet $L$-values with complex characters.} and McPhedran \cite{ZM}, who asked whether or not $S(1,0,36;t)$ reduces to known quantities.  See Section \ref{sec: irreducible values of S} for the proof of \eqref{S(1,0,36) intro formula}, and for additional examples.

\section{Review of Ramanujan's formulas}\label{sec review}
We begin with a brief, but in-depth review of
Ramanujan's formulas.  Suppose that \eqref{Ramanujan type series} holds for certain values
of $(a, b, z)$ and $s$.  Let $y_0(z)$ denote the following ${_3{\ff}_2}$ function:
\begin{equation}\label{y0 definition}
y_0(z)=\hpg32{s,\,\frac{1}{2},\,1-s}{1,1}{\,z}=\sum_{n=0}^{\infty}
\frac{(s)_n \left(\frac{1}{2} \right)_n (1-s)_n}{(1)_n^3} z^n.
\end{equation}
We parameterize $(a, b, z)$ in terms of $q$.
Suppose that $q$ and $z$ are related by the differential equation:
\begin{equation}\label{dq-dz}
\frac{d q}{d z}=\frac{q}{y_{0}(z) z\sqrt{1-z}}.
\end{equation}
It is possible to express $z$ in terms of $q$ by integrating and then inverting \eqref{dq-dz}. The inverse expressions are related to theta functions when $s\in\{\frac{1}{2},\frac13,\frac14,\frac16\}$ (we use \eqref{param s=1/2} when $s=\frac12$).  The formulas for $a$ and $b$ are given by:
\begin{align}\label{ecu-a-b}
a=\frac{1}{\pi y_{0}(z)} \left( 1 +  \frac{\ln|q|}{y_{0}(z)} q
\frac{dy_{0}(z)}{dq} \right),&&b=- \frac{\ln|q|}{\pi} \sqrt{1-z}.
\end{align}
The parameterizations can be verified by substituting them into \eqref{Ramanujan
type series}.  It is a deep fact that $(a, b, z)$ are \textit{algebraic}, whenever $q=e^{2\pi i(x_1+i\sqrt{|x_2|})}$ with $(x_1,x_2)\in\mathbb{Q}^2$, and $s\in\{\frac12,\frac13,\frac14,\frac16\}$. The algebraic numbers are usually complicated,
however rational evaluations occur in some instances.

\begin{proposition}\label{teor-y-rama}
Assume that $(a, b, z)$ and $q$ are related by \eqref{dq-dz}
and \eqref{ecu-a-b}. Suppose that $f(z)$ is a differentiable
function, and let
\begin{equation*}
\phi_f(q)=\frac{f(z)}{y_{0}(z)}.
\end{equation*}
Then
\begin{equation}\label{completa-zq} a f(z)+b z \frac{d
f(z)}{dz}=\frac{1}{\pi} \left(\phi_f(q) - \ln|q| \, q \frac{d
\phi_f(q)}{dq} \right).
\end{equation}
\end{proposition}
\begin{proof}
From the right-hand side we have
\begin{align*}
\frac{1}{\pi} \left( \phi_f(q) - \ln|q| \, q \frac{d \phi_f(q)}{dq}
\right)&=\frac{1}{\pi}\left( \frac{f(z)}{y_0(z)} - \ln|q| \, q
\frac{d}{dq}\frac{f(z)}{y_0(z)} \right)\\
&=
\frac{1}{\pi} \frac{f(z)}{y_0(z)}- \ln|q| \frac{q}{\pi y_0^2(z)} \left( y_0(z)\frac{d f(z)}{dq}-f(z)\frac{d y_0(z)}{dq} \right) \notag\\
&= \frac{1}{\pi} \left( \frac{1}{y_0(z)}+\frac{\ln|q|}{y_0^2(z)} q
\frac{d y_0(z)}{dq} \right) f(z)\\
&\qquad- \left( \frac{\ln|q|}{\pi
y_0(z)}\frac{q }{z} \frac{dz}{dq} \right) z\frac{d f(z)}{dz}\\
&=a f(z)+b z \frac{d f(z)}{dz}. \nonumber
\end{align*}
The final step follows from \eqref{ecu-a-b}.
\end{proof}

Proposition \ref{teor-y-rama} allows us to insert a factor of
$(a+b n)$ into a power series.  For example, if $f(z)=y_{0}(z)$,
then $\phi_f(q)=1$.  We have
\begin{equation*}
1=\frac{1}{y_{0}(z)}\sum_{n=0}^{\infty}\frac{(s)_n(\frac{1}{2})_n(1-s)_n}{(1)_n^3}z^n.
\end{equation*}
By Proposition \ref{teor-y-rama} this becomes
\begin{equation*}
\frac{1}{\pi} \left(1 - \ln|q| \, q
\frac{d}{dq}\right)\cdot1=\left(a+b
z\frac{d}{dz}\right)\cdot\sum_{n=0}^{\infty}\frac{(s)_n(\frac{1}{2})_n(1-s)_n}{(1)_n^3}z^n,
\end{equation*}
hence
\begin{equation*}
\frac{1}{\pi}
=\sum_{n=0}^{\infty}\frac{(s)_n(\frac{1}{2})_n(1-s)_n}{(1)_n^3}(a+b
n)z^n.
\end{equation*}
More difficult cases require us to expand $f(z)/y_0(z)$ in a $q$-series, before applying Proposition \ref{teor-y-rama}.

\section{Completing the hypergeometric function}\label{sec:completing the hypergeometric}
In this section we introduce the idea of \textit{completing a
hypergeometric function}.  Hypergeometric functions are typically defined by an infinite series, and analytically continued
to a slit plane via integral formulas. To complete a
hypergeometric function, let $n\mapsto n+x$ in the series
definition, and extend the sum over $n\in\mathbb{Z}$.  Consider $y_{0}(z)$, defined in \eqref{y0 definition}, as an example. The completed version of $y_{0}(z)$ is a formal sum
\begin{equation}\label{Yx formal sum}
\sum_{n\in\mathbb{Z}}\frac{(s)_{n+x}\left(\frac{1}{2}\right)_{n+x}(1-s)_{n+x}}{(1)_{n+x}^3}z^{n+x},
\end{equation}
which involves powers of $z$ and $z^{-1}$.  To avoid divergence issues, consider the positive ($n\ge 0$) and negative ($n<0$) halves of the sum as hypergeometric functions.  This transforms \eqref{Yx formal sum} into a well-defined function:
\begin{equation}\label{Yx definition}
\begin{split}
Y_{x}(z):=&z^x\frac{\left(\frac{1}{2}\right)_x \left(
  1 - s\right)_x \left(s\right)_x}{(1)_x^3}
  \hpg43{1,\,\frac12+x,\,1 - s + x,\,s + x}{1+x,1+x,1+x}{\,z}\\
   &- \frac{2x^3 z^{x-1}}{s(1-s)} \frac{\left( -\frac12
\right)_x(s-1)_x(-s)_x }{(1)_x^3}\hpg43{1,\,1-x,\,1 - x,\,1- x}{\frac{3}{2} - x, 2 - s - x,1 + s - x}{\,\frac{1}{z}}
%
\end{split}
\end{equation}
which is certainly analytic for $z\in\mathbb{C}\setminus\mathbb{R}$ (the ${_4{\ff}_3}$ functions and $z^x$ have branch cuts on the real axis).  From \eqref{Yx formal sum} it is obvious that $Y_{x}(z)$ is periodic in $x$:
\begin{equation*}
Y_{x}(z)=Y_{x+1}(z).
\end{equation*}
This property extends to \eqref{Yx definition}, because ${_4{\ff}_3}$ functions obey recurrences in their parameters, regardless of $z$.  Below we prove that $Y_x(z)$ equals a trigonometric polynomial in $x$.  This is the key result which enables us to sum up the companion series in Theorem \ref{main theorem}.

\begin{lemma}\label{Lemma: reduction of Yx(z)}  Suppose that $s\in(0,1)$ and $z\not\in\{0,1\}$.
There exist functions $u:=u(z)$ and $v:=v(z)$ which are independent
of $x$, such that
\begin{equation}\label{T1 reduction}
Y_x(z) = y_0(z) \, \frac{e^{i \pi x} \sin^2 \pi s}{\cos \pi x
(\cos^2 \pi x - \cos^2 \pi s)} \left(-u+(u+1)\cos 2 \pi x  - i v
\sin 2 \pi x \right).
\end{equation}
\end{lemma}
\begin{proof}
Consider the
Picard-Fuchs operator which annihilates $y_0(z)$.  Let
\begin{equation}\label{picard-fu}
\begin{split}
P:=&\left( z\frac{d}{dz} \right)^3 - z \left(
z\frac{d}{dz}+\frac{1}{2} \right)\left( z\frac{d}{dz}+s
\right)\left( z\frac{d}{dz}+1-s \right).
\end{split}
\end{equation}
If convergence issues are ignored, then it is easy to show that $P$ also annihilates \eqref{Yx formal sum}.  This allows us to extrapolate
\begin{equation}\label{Yx annihilated}
P Y_{x}(z)=0.
\end{equation}
It is possible to prove \eqref{Yx annihilated} using standard rules for differentiating hypergeometric functions, but we leave this as an exercise.  Since $P$ annihilates $Y_x(z)$, the function has the form:
\begin{equation}\label{Yx decomp}
Y_{x}(z)=m_0(x)y^{(0)}(z)+m_1(x)y^{(1)}(z)+m_2(x)y^{(2)}(z),
\end{equation}
where each $y^{(i)}$ is a linearly independent solution of $Py=0$.  The linear independence property implies that $m_{i}(x)=m_{i}(x+1)$ for all $i$ (if the $m_{i}$'s are not periodic, then $Y_x(z)-Y_{x+1}(z)=0$ leads to a linear dependence between $y^{(i)}$'s).  We derive formulas for $m_i(x)$ below.

Suppose that $s\in(0,1)$, and that $z$ is not a singular point of $Y_x(z)$ (we exclude $z=0$ and $z=1$).  Since $Y_{x}(z)=Y_{x+1}(z)$, we assume without loss of generality that $\Re(x)\in[0,1)$.  We claim that $Y_x(z)$ is meromorphic in $x$, with simple poles at $x\in\{s,\frac12,1-s\}$.  To prove this, first recall that ${_4{\ff}_3}\left(a_1,a_2,a_3,a_4; b_1, b_2, b_3;z\right)$, is meromorphic with respect to each $b_i$, provided $z$ is not a singular point \cite{DL}.  Poles occur if $b_i\in\{0,-1,-2,\dots\}$.  Since $(\Re(x),s)\in[0,1)\times(0,1)$, it is easy to check that $\{1+x,\frac{3}{2}-x,2-s-x,1+s-x\}\bigcap\{0,-1,-2,\dots\}=\emptyset$, thus the ${_4{\ff}_3}$ functions in \eqref{Yx definition} do not contribute poles.  Next observe
\begin{align*}
\frac{\left( -\frac12\right)_x(s-1)_x(-s)_x}{(1)_x^3}=&\frac{\Gamma(-\frac12+x)\Gamma(s-1+x)\Gamma(-s+x)}{\Gamma\left(\frac12\right)\Gamma(-s)\Gamma(s-1)\Gamma^3(1+x)},\\
\frac{\left(\frac{1}{2}\right)_x \left(
1 - s\right)_x \left(s\right)_x}{(1)_x^3}=&\frac{\Gamma(\frac12+x)\Gamma(1-s+x)\Gamma(s+x)}{\Gamma(\frac12)\Gamma(1-s)\Gamma(s)\Gamma^3(1+x)}.
\end{align*}
The first ratio of Pochhammer symbols contributes simple poles when $x\in\{s,\frac{1}{2},1-s\}$, and the second ratio of Pochammer symbols is analytic for $(\Re(x),s)\in[0,1)\times(0,1)$.  By the linear independence argument above, we conclude that $m_{i}(x)$ is at worst meromorphic with simple poles when $x\in\{s,\frac12,1-s\}$.

Now we show that $m_i(x)=O(|\Im(x)|^{-3/2})$ when $|\Im(x)|$ is sufficiently large.  Let $z\in[\epsilon,1-\epsilon]$, for some small $\epsilon>0$. Note that $|z^x|=|z|^{\Re(x)}<1$.  Formula \eqref{Yx definition} becomes
\begin{align*}
|Y_{x}(z)|<&\left|\frac{\left(\frac{1}{2}\right)_x \left(
  1 - s\right)_x \left(s\right)_x}{(1)_x^3} \hpg43{1,\,\frac{1}{2}+x,\,1 - s + x,\,s + x}{1+x,1+x,1+x}{\,z}
   \right.\\
   &\qquad\left.- \frac{2x^3 z^{-1}}{s(1-s)} \frac{\left( -\frac12
\right)_x(s-1)_x(-s)_x }{(1)_x^3}
\hpg43{1,\,1-x,\,1 - x,\,1- x}{\frac{3}{2}-x,2-s-x,1+s-x}{\,\frac{1}{z}}
\right|.
\end{align*}
The right-hand side of the inequality vanishes when $|\Im(x)|\mapsto \infty$.  To see this, use the estimates
\begin{align*}
\hpg43{1,\,\frac{1}{2}+x,\,1 - s + x,\,s + x}{1+x,1+x,1+x}{\,z}&\approx\hpg10{1}{}{\,z}=\frac{1}{1-z}\\
 \hpg43{1,\,1-x,\,1 - x,\,1- x}{\frac{3}{2}-x,2-s-x,1+s-x}{\,\frac{1}{z}}&\approx\hpg10{1}{}{\,\frac{1}{z}}=\frac{z}{z-1},\\
  \frac{(1-s)_x\left(\frac12\right)_x(s)_x}{(1)_x^3}&\approx\frac{\sin\pi s}{(\pi i \Im(x))^{3/2}},\\
  \frac{2x^3}{s(1-s)} \frac{\left( -\frac12
\right)_x(s-1)_x(-s)_x }{(1)_x^3}&\approx-\frac{\sin\pi s}{(\pi i \Im(x))^{3/2}},
\end{align*}
which are valid when $|\Im(x)|$ is large.  Thus if $|\Im(x)|$ is sufficiently large (which rules out the possibility of $x$ lying in a neighborhood of the points $\{s,\frac12,1-s\}$), then $Y_x(z)=O(|\Im(x)|^{-3/2})$.  The estimate holds uniformly for
$z\in[\epsilon,1-\epsilon]$, so a linear independence argument suffices to show that $m_i(x)=O(|\Im(x)|^{-3/2})$ for each $i$.

We have proved that $m_i(x)$ is periodic and meromorphic, with (possible) simple poles at $x\in\{s,\frac12,1-s\}$.  If $|\Im(x)|$ is sufficiently large, then $m_i(x)=O(|\Im(x)|^{-3/2})$.  We conclude that
\begin{equation*}
e^{-i \pi x} \cos \pi x
(\cos^2 \pi x - \cos^2 \pi s)m_i(x)
\end{equation*}
is analytic for $\Re(x)\in[0,1)$.  This new function has period $1$, so it is also analytic on $\mathbb{C}$.  The function is majorized by $O\left(|\Im(x)|^{-3/2}e^{4\pi |\Im(x)|}\right)$ for $|\Im(x)|$ sufficiently large.  Therefore the function has a Fourier series which terminates:
\begin{equation*}
e^{-i \pi x} \cos \pi x
(\cos^2 \pi x - \cos^2 \pi s)m_i(x)=a_i^{(0)}+a_i^{(1)}\cos(2\pi x)+a_i^{(2)} \sin(2\pi x).
\end{equation*}
After collecting constants in \eqref{Yx decomp}, and noting that $Y_{0}(z)=y_{0}(z)$, we conclude that $Y_x(z)$ has the form given in \eqref{T1 reduction}.
\end{proof}

Now let $y_x(z)$ denote
the positive half ($n\ge 0$) of the completed hypergeometric function:
\begin{equation}\label{y-x}
\begin{split}
y_x(z)&:=\sum_{n=0}^{\infty} \frac{ \left( \frac{1}{2} \right)_{n+x}
\left( s \right)_{n+x}\left( 1-s \right)_{n+x}}{ (1)_{n+x}^3
}z^{n+x}\\
&=z^x\frac{\left(\frac{1}{2}\right)_x \left(
  1 - s\right)_x \left(s\right)_x}{(1)_x^3}
  \hpg43{1,\,\frac{1}{2}+x,\,1 - s + x,\,s + x}{1+x,1+x,1+x}{\,z}.
\end{split}
\end{equation}
The first author calls this an \textit{extended hypergeometric series} \cite{G10Rama}.
Since $y_x(z)$ is analytic in a neighborhood
of $x=0$, we have a Maclaurin series of the form
\begin{equation}\label{y-x over y-0 series}
\frac{y_{x}(z)}{y_0(z)}=1+\phi_1(q) x+\phi_2(q) x^2+\phi_{3}(q)x^3 +O(x^4),
\end{equation}
where $z$ and $q$ are related by \eqref{dq-dz}.  Since $y_x(z)/y_0(z)$ is non-holomorphic in $z$, we expect each $\phi_i(q)$ to
be non-holomorphic in $q$.

\begin{theorem}\label{main theorem}
Assume that $s\in(0,1)$, $z\not\in\{0,1\}$, and let $\phi_i(q)$ be as in \eqref{y-x over y-0 series}.  Then
\begin{equation}\label{pre-Sun in terms of phii's}
\begin{split}
\frac{1}{\pi y_{0}(z)}\sum_{n=1}^{\infty}&\frac{(1)_n^3}{(s)_n\left(\frac{1}{2}\right)_n(1-s)_n}\frac{z^{-n}
}{n^3}\\
&=\pi^2 i \csc^2\left(\pi s\right) -
 \frac{\pi}{3} \left(1 + 3 \csc^2(\pi s) \right)\phi_1(q)-
 i \phi_2(q)+ \frac{1}{\pi}\phi_3(q).
   \end{split}
\end{equation}
By Proposition \ref{teor-y-rama}, we also have
\begin{equation}\label{Sun in terms of phii's}
\begin{split}
\sum_{n=1}^{\infty}&\frac{(1)_n^3}{(s)_n\left(\frac{1}{2}\right)_n(1-s)_n}\frac{(a-b
n)}{n^3}z^{-n}\\
&=\pi^2 i \csc^2\left(\pi s\right) -
 \frac{\pi}{3} \left(1 + 3 \csc^2(\pi s) \right)\left(\phi_1(q) -
    q\log|q|\frac{d\phi_1(q)}{d q}\right) \\
    &\qquad-
 i \left(\phi_2(q) - q\log|q|\frac{d \phi_2(q)}{d q}\right) + \frac{1}{\pi}\left(\phi_3(q) -
   q\log|q|\frac{d\phi_3(q)}{d q}\right).
   \end{split}
\end{equation}
The sums in \eqref{pre-Sun in terms of phii's} and \eqref{Sun in terms
of phii's} diverge if $|z|<1$, however the identities remain valid
when ${_4{\ff}_3}$ and ${_5{\ff}_4}$ functions are substituted.
\end{theorem}
\begin{proof} From \eqref{Yx definition} and \eqref{y-x} we see that
\begin{equation*}
Y_{x}(z)=y_{x}(z)+O(x^3).
\end{equation*}
This is sufficient to determine $u$ and $v$ in \eqref{T1 reduction}.  From \eqref{y-x} we find
\begin{equation*}
\frac{y_{x}(z)}{y_{0}(z)}=1+\phi_1(q)x+\phi_2(q)x^2+\phi_3(q)x^3+O(x^4).
\end{equation*}
By \eqref{T1 reduction} we also have
\begin{equation}\label{Yx taylor series}
\begin{split}
\frac{Y_{x}(z)}{y_{0}(z)}=&1+i \pi  (1-2 v) x+\pi^2 \left(-2-2 u+2 v+\csc^2(\pi  s)\right) x^2\\
&-\frac{i\pi^3}{3} \left(5+6 u-4 v+(-3+6 v) \csc^2(\pi s)\right)
x^3+O(x^4),
\end{split}
\end{equation}
where $s$ and $z$ satisfy the appropriate restrictions.
The Taylor coefficients of $Y_{x}(z)$ and $y_x(z)$ agree up to order $x^2$.
This leads to a pair of equations
\begin{align*}
\phi_{1}(q)&=i \pi  (1-2 v)\\
\phi_{2}(q)&=\pi^2 \left(-2-2 u+2 v+\csc^2(\pi  s)\right),
\end{align*}
from which it is easy to solve for $u$ and $v$.

The companion series arises from the $x^3$ coefficient of $Y_x(z)$.  By \eqref{Yx definition} and \eqref{y-x} we have
\begin{equation*}
\begin{split}
\frac{1}{y_0(z)}\sum_{n=1}^{\infty}\frac{(1)_n^3}{(s)_n\left(\frac{1}{2}\right)_n(1-s)_n}\frac{z^{-n}}{n^3}
   =&\frac{1}{y_0(z)}\frac{2z^{-1}}{s(1-s)}\hpg43{1,\,1,\,1,\,1}{\frac{3}{2},2-s,1+s}{\,\frac{1}{z}}
   \\
   =&\lim_{x\rightarrow 0}\left(\frac{y_{x}(z)-Y_{x}(z)}{y_0(z)~x^3}\right)\\
   =&\phi_3(q)+\frac{i\pi^3}{3}\left(5+6 u-4 v+(-3+6 v) \csc^2(\pi s)\right).
\end{split}
\end{equation*}
We recover \eqref{pre-Sun in terms of phii's} by eliminating $u$ and $v$.
\end{proof}
Despite the fact that \eqref{pre-Sun in
terms of phii's} and \eqref{Sun in terms of phii's} hold for many
values of $s$, we have only been able to evaluate $\phi_i(q)$
if $s\in\{\frac{1}{2},\frac{1}{3},\frac{1}{4}\}$. We prove formulas
for $\phi_i(q)$ below.
\begin{theorem}\label{thm extended coeff} Suppose that $q$ lies in a neighborhood of zero. When
$s=\frac12$:
\begin{align}
\phi_1(q)=&\ln q, \\
\phi_2(q)=&\frac{1}{2}\ln^2 q +
\frac{\pi^2}{2}, \\
\phi_3(q)=&\frac{1}{6} \ln^3 q +\frac{\pi^2}{2}\ln
q-6\zeta(3)-16\sum_{n=1}^{\infty} \frac{\sigma_3(n)}{n^3} q^n+4
\sum_{n=1}^{\infty} \frac{\sigma_3(n)}{n^3} q^{4n}.
\end{align}
When $s=\frac13$:
\begin{align}
\phi_1(q)=&\ln q, \\
\phi_2(q)=&\frac{1}{2}\ln^2 q +
\frac{2\pi^2}{3}, \\
\phi_3(q)=&\frac{1}{6} \ln^3 q +\frac{2\pi^2}{3}\ln
q-10\zeta(3)-30\sum_{n=1}^{\infty} \frac{\sigma_3(n)}{n^3} q^n+10
\sum_{n=1}^{\infty} \frac{\sigma_3(n)}{n^3} q^{3n}.
\end{align}
When $s=\frac14$:
\begin{align}
\phi_1(q)=&\ln q, \\
\phi_2(q)=&\frac{1}{2}\ln^2 q +
\pi^2, \\
\phi_3(q)=&\frac{1}{6} \ln^3 q +\pi^2\ln
q-20\zeta(3)-80\sum_{n=1}^{\infty} \frac{\sigma_3(n)}{n^3} q^n+40
\sum_{n=1}^{\infty} \frac{\sigma_3(n)}{n^3} q^{2n}.
\end{align}
\end{theorem}
\begin{proof}  The essential idea is to apply the
Picard-Fuchs operator which annihilates $y_0(z)$.  Recall that $P$ is defined in \eqref{picard-fu}.  It was proved in
\cite[Prop. 2.2]{guilleraMat}, that
\begin{equation}\label{Pw}
P y_x(z) = \frac{(1-s)_x\left(\frac{1}{2}\right)_x(s)_x}{(1)_x^3}z^x
x^3=x^3+O(x^4).
\end{equation}
When $x=0$, we immediately obtain the homogeneous differential
equation $P y_0(z)=0$.  If $y_x(z)$ is expanded in a Maclaurin
series with respect to $x$, then by \eqref{y-x over y-0 series} we
have $P (y_0(z) \phi_1(q))=0$ and $P (y_0(z) \phi_2(q))=0$.
Appealing to \cite[Lemma 1]{yang}, we see that
\begin{equation}\label{yifan12}
\left( q \frac{d}{dq} \right)^3 \phi_1(q)=0, \qquad \left( q
\frac{d}{dq} \right)^3 \phi_2(q)=0,
\end{equation}
and integrating gives
\begin{align}
\phi_1(q)=&\alpha_0+\alpha_1 \ln q+\alpha_2 \ln^2 q,\label{phi1 undet}\\
\phi_2(q)=&\beta_0+\beta_1 \ln q+\beta_2 \ln^2 q,\label{phi2
undet}
\end{align}
where the $\alpha_{i}$'s and $\beta_{i}$'s are undetermined
constants. Examining the $x^3$ coefficient of $y_x(z)$, leads to the
inhomogeneous differential equation $P [ y_0(z) \phi_3(q) ]=1$. By
\cite[Lemma 1]{yang} and \cite[Iden. 2.33]{guilleraMat}, we find
that
\begin{equation}\label{yifan3}
\left( q \frac{d}{dq} \right)^3 \phi_3(q)= \sqrt{1-z}~y_0^2(z).
\end{equation}
In order to solve \eqref{yifan3}, and to determine the constants
in \eqref{phi1 undet} and\eqref{phi2 undet}, it is necessary to
specify the value of $s$.

Suppose that $q$ lies in a neighborhood of zero.  When $s=\frac12$
we have $\sqrt{1-z}=1-2\lambda(q)$, where
$\lambda(q)=\theta_2^4(q)/\theta_3^4(q)$ is the elliptic lambda
function \cite[Sect. 2.5]{guilleraMat}.  By standard theta function
inversion formulas, we also have
\begin{equation}\label{y0 in terms of theta3}
 y_0(z)=\theta_3^4(q).
\end{equation}
Identity \eqref{y0 in terms of theta3} \textit{does not} hold for
$|q|<1$.  For instance, if $q$ is close to $1$ we have to replace
\eqref{y0 in terms of theta3} with
$y_0(z)=\frac{\log^2(q)}{\pi^2}\theta_3^4(q)$.  For $|q|$
sufficiently small
\begin{equation*}
\begin{split}
y_0^2(z)\sqrt{1-z}&=\theta_3^8(q)-2\theta_3^4(q)\theta_2^4(q)\\
&=1-16\sum_{n=1}^{\infty} \sigma_3(n) q^n+16^2 \sum_{n=1}^{\infty}
\sigma_3(n) q^{4n},
\end{split}
\end{equation*}
where the second equality follows from \cite[~pg. 126, Entry
13]{Be3}. Integrating \eqref{yifan3} gives
\begin{equation}\label{s12phi3-gammas}
\begin{split}
\phi_3(q)=&\gamma_0 + \gamma_1 \ln q + \gamma_2 \ln^2 q +
\frac{1}{6}
\ln^3 q\\
& - 16 \sum_{n=1}^{\infty} \sigma_3(n)
\frac{q^n}{n^3}+4\sum_{n=1}^{\infty} \sigma_3(n)
\frac{q^{4n}}{n^3},
\end{split}
\end{equation}
where the $\gamma_{i}$'s are constants.

There are nine constants left to calculate. Let $q$ tend to zero in
\eqref{y-x over y-0 series}. Since $z$ has a $q$-series of the form
$z=64q +O(q^2)$, it follows that $z\approx64q$ when $q$ approaches
zero.  In a similar manner we find that $y_0(z)\approx 1$.  By
\eqref{y-x over y-0 series} we have
\begin{align}
q^{-x} y_x(z)&=q^{-x} y_0(z)
\left(1+\phi_1(q)x+\phi_2(q)x^2+\phi_3(q)x^3+O(x^4)\right)\notag\\
&\approx
q^{-x}\left(1+\phi_1(q)x+\phi_2(q)x^2+\phi_3(q)x^3+O(x^4)\right)\label{constantcalc1}.
\end{align}
From the definition of $q^{-x} y_x(z)$, we calculate
\begin{align}
q^{-x}y_x(z)=&q^{-x}
z^{x}\frac{\left(\frac{1}{2}\right)_x^3}{(1)_x^3}\left(1+\sum_{n=1}^{\infty}
z^n \frac{ \left( \frac{1}{2} + x \right)_{n}^3}{(1+x)_{n}^3}\right)\notag\\
\approx&
64^{x}\frac{\left(\frac{1}{2}\right)_x^3}{(1)_x^3}(1+0)\label{constantcalc2}
\end{align}
Compare the Maclaurin series coefficients of \eqref{constantcalc1}
and \eqref{constantcalc2} in $x$, $x^2$, and $x^3$. Since
\eqref{constantcalc2} is holomorphic at $x=0$, it follows that
\eqref{constantcalc1} is holomorphic at $x=0$ as well.  Since $q$
tends to zero, this implies that the powers of $\log(q)$ must drop
out of the series obtained from \eqref{constantcalc1}. Comparing
coefficients then provides sufficiently many relations to determine
the values of $\alpha_{i}$, $\beta_{i}$, and $\gamma_{i}$
explicitly.  The cases when $s=\frac13$ and $s=\frac14$ require
analogous arguments, using appropriate theta functions from
\cite{Be5}.
\end{proof}

The method fails when $s=\frac16$, because of our inability to solve \eqref{yifan3}.  The calculation is difficult because Ramanujan's theory of signature-$6$ modular equations is incomplete, and as a result it seems to be impossible to find a nice $q$-series expansion for $\sqrt{1-z}~y_0^2(z)$.  Notice that \eqref{yifan3} is equivalent to
\begin{equation}\label{s=1/6 failure}
\left(q\frac{d}{d q}\right)^3\phi_3(q)=\frac{1-504\sum_{n=1}^{\infty}\frac{n^5q^n}{1-q^n}}{\sqrt{1+240\sum_{n=1}^{\infty}\frac{n^3 q^n}{1-q^n}}}.
\end{equation}
If we could obtain a reasonable expression for $\phi_3(q)$, then it might be possible to evaluate a companion series with $s=\frac16$.  Experimental searches failed to turn up any interesting identities, so we suspect that the task is impossible.

\section{Explicit Formulas}\label{sec: explicit formulas}
Now we prove companion series evaluations.  Proposition
\ref{Proposition:Companion series reduction} reduces every companion
series to elementary constants and values of the following special
function:
\begin{equation}\label{F(q) definition}
\begin{split}
F(q):=&-\frac{\log^3|q|}{3\pi}+\frac{120}{\pi}\zeta(3)+\frac{240}{\pi}\sum_{j=1}^{\infty}\Li_{3}(q^j)-\log|q^j|\Li_{2}(q^j).
\end{split}
\end{equation}
Notice that $F(q)$ is closely related to the elliptic trilogarithm \cite{ZG}.
Set $q=e^{2\pi i \tau}$, with
$\tau=x+i y$, and $y>0$. In Proposition \ref{proposition: F in terms of S} we
prove
\begin{equation}\label{Re(F(q)) evaluation}
\Re\left(F(q)\right)=\frac{120y^3}{\pi^2}S\left(1,2x,x^2+y^2;2\right).
\end{equation}
It is easy to see that $F(q)$ is real-valued if $q\in(-1,1)$, so \eqref{Re(F(q)) evaluation} becomes a formula for $F(q)$ whenever $x\in\mathbb{Z}/2$.  Glasser and Zucker proved that $S(A,B,C;t)$ reduces to Dirichlet
$L$-values quite often. This leads to $65$ evaluations of $F(q)$, when $x=0$ and
$y^2\in\mathbb{N}$. For instance, when $(x,y)=(0,\sqrt{7})$, we have
\begin{equation*}
F\left(e^{-2\pi\sqrt{7}}\right)=175\sqrt{7}L_{-7}(2).
\end{equation*}
Various additional values of $F(q)$ are provided in Table \ref{Table:F(q) formulas}.
The formulas in Theorems \ref{Theorem: proving Sun's formulas} and
\ref{Theorem: Divergent rational cases} are proved by evaluating linear combinations of $F(q)$'s.

\begin{proposition}\label{Proposition:Companion series reduction} Suppose that $q$ lies in a neighborhood of zero.  When $s=\frac12$:
\begin{equation}\label{5F4 1/2 case in terms of F}
\begin{split}
\sum_{n=1}^{\infty}\frac{(1)_n^3}{\left(\frac{1}{2}\right)_n^3}\frac{(a-b
n)}{n^3}z^{-n}
=&-\frac{1}{15}F(q)+\frac{1}{60}F(q^4)\\
&+\frac{\log(q)^3}{6 \pi }-\frac{\log(q)^2 \log|q|}{2 \pi }+\frac{\log|q|^3}{3 \pi }\\
&-\frac{i}{2}\log(q)^2+i\log(q)\log|q|\\
&-\frac{5}{6}\pi\log(q)+\frac{5}{6}\pi\log|q|+\frac{i \pi ^2}{2}.
\end{split}
\end{equation}
When $s=\frac13$:
\begin{equation}\label{5F4 1/3 case in terms of F}
\begin{split}
\sum_{n=1}^{\infty}\frac{(1)_n^3}{\left(\frac{1}{3}\right)_n\left(\frac{1}{2}\right)_n\left(\frac{2}{3}\right)_n}\frac{(a-b
n)}{n^3}z^{-n}
=&-\frac{1}{8}F(q)+\frac{1}{24}F(q^3)\\
&+\frac{\log^3(q)}{6 \pi }-\frac{\log^2(q)\log|q|}{2 \pi}+\frac{\log^3|q|}{3 \pi }\\
&-\frac{i}{2}\log^2(q)+i \log(q)
\log|q|\\
&-\pi  \log(q)+\pi\log|q|+\frac{2 i \pi ^2}{3}.
\end{split}
\end{equation}
When $s=\frac14$:
\begin{equation}\label{5F4 1/4 case in terms of F}
\begin{split}
\sum_{n=1}^{\infty}\frac{(1)_n^3}{\left(\frac{1}{4}\right)_n\left(\frac{1}{2}\right)_n\left(\frac{3}{4}\right)_n}\frac{(a-b
n)}{n^3}z^{-n}
=&-\frac{1}{3}F(q)+\frac{1}{6}F\left(q^2\right)\\
&+\frac{\log^3(q)}{6 \pi }-\frac{\log^2(q)\log|q|}{2 \pi }+\frac{\log^3|q|}{3 \pi }\\
&-\frac{1}{2} i \log^2(q)+i\log(q)\log|q|\\
&-\frac{4}{3}\pi\log(q)+\frac{4}{3}\pi\log|q|+i \pi ^2.
\end{split}
\end{equation}
\end{proposition}
\begin{proof}  Proofs follow from combining Theorems
\ref{main theorem} and \ref{thm extended coeff}.  In particular, we
obtain formulas \eqref{5F4 1/2 case in terms of F} through
\eqref{5F4 1/4 case in terms of F}, by substituting the results of
Theorem \ref{thm extended coeff} into \eqref{Sun in terms
of phii's}.
\end{proof}

\begin{proposition}\label{proposition: F in terms of S} Let $q=e^{2\pi i \tau}$, with $\tau=x+i y$, and $y>0$. Then
\begin{equation}\label{final F(q) intermed evaluation}
\begin{split}
F(q)=\frac{120y^3}{\pi^2}S(1,2x,x^2+y^2;2)+\frac{60i}{\pi^2}\sum_{\substack{n,k\\n\ne0}}\frac{(k+n
x)\left((k+n x)^2+3 n^2 y^2\right)}{n^3\left((k+n x)^2+n^2
y^2\right)^2}.
\end{split}
\end{equation}

If $x\in\mathbb{Z}/2$ and $y>0$, then
\begin{equation}\label{final F(q) evaluation}
\begin{split}
F(q)=\frac{120y^3}{\pi^2}S(1,2x,x^2+y^2;2).
\end{split}
\end{equation}

If $2x/(x^2+y^2)\in\mathbb{Z}$ and $y>0$, then
\begin{equation}\label{final F(q) integer case}
\begin{split}
F(q)=\frac{120y^3}{\pi^2}S(1,2x,x^2+y^2;2)+\frac{4i\pi^2}{3}x\left(\frac{x^2+3y^2}
{(x^2+y^2)^2}+x^2+3y^2-5\right).
\end{split}
\end{equation}
\end{proposition}
\begin{proof}By \eqref{F(q) definition} we obtain
\begin{equation*}
\begin{split}
 F(q)=&\frac{8\pi^2}{3}\left(\Im\tau\right)^3+\frac{120}{\pi}\sum_{n=1}^{\infty}\left(\frac{1}{n^3}+\frac{2}{n^3}\sum_{j=1}^{\infty}q^{j
n}+\frac{4\pi\Im(\tau)}{n^2}\sum_{j=1}^{\infty}j q^{j n}\right)\\
=&\frac{8\pi^2}{3}\left(\Im\tau\right)^3+\frac{120}{\pi}\sum_{n=1}^{\infty}\left(\frac{1}{n^3}
\frac{1+q^n}{1-q^n}+\frac{4\pi
\Im(\tau)}{n^2}\frac{q^n}{(1-q^n)^2}\right)\\
=&\frac{8\pi^2}{3}\left(\Im\tau\right)^3+\frac{60}{\pi}\sum_{\substack{n=-\infty\\n\ne0}}^{\infty}\left(\frac{i\cot(\pi
n \tau)}{n^3} -\frac{\pi \Im(\tau)\csc^2(\pi n \tau)}{n^2}\right).
\end{split}
\end{equation*}
Substitute the partial fractions decompositions:
\begin{align*}
\cot(\pi
n\tau)=\frac{1}{\pi}\sum_{k=-\infty}^{\infty}\frac{1}{k+\tau
n},&&\pi\csc^2(\pi
n\tau)=\frac{1}{\pi}\sum_{k=-\infty}^{\infty}\frac{1}{(k+\tau n)^2},
\end{align*}
to obtain
\begin{equation}\label{main F(q) partial fractions}
F(q)=\frac{8\pi^2}{3}\left(\Im\tau\right)^3+\frac{60}{\pi^2}\sum_{\substack{n,k=-\infty\\n\ne
0}}^{\infty}\frac{i}{n^3(k+n\tau)}-\frac{\Im(\tau)}{n^2(k+n\tau)^2}.
\end{equation}
Formula \eqref{final F(q) intermed evaluation} follows from setting $\tau=x+i y$, and then isolating the real and imaginary parts of the function.
We complete the proof of \eqref{final F(q) evaluation} by noting that $F(q)$ is real valued whenever $x\in\mathbb{Z}/2$.

To complete the proof of \eqref{final F(q) integer case} we need to evaluate the following sum:
\begin{equation*}
T(x,y):=\sum_{\substack{n,k\\n\ne0}}\frac{(k+n
x)\left((k+n x)^2+3 n^2 y^2\right)}{n^3\left((k+n x)^2+n^2
y^2\right)^2}.
\end{equation*}
Extract the $k=0$ term, to obtain
\begin{equation*}
T(x,y)=\frac{\pi^4}{45}\frac{x (x^2 + 3 y^2)}{(x^2 + y^2)^2}+\sum_{\substack{k\\k\ne0}}\sum_{\substack{n\\n\ne0}}\frac{(k+n
x)\left((k+n x)^2+3 n^2 y^2\right)}{n^3\left((k+n x)^2+n^2
y^2\right)^2}.
\end{equation*}
When $k\ne 0$ the inner sum can be evaluated by the residues method.  \texttt{Mathematica} produces the following formula:
\begin{align*}
\sum_{\stackrel{\scriptstyle n =-\infty}{n \neq0}}^{\infty}& \frac{(k+n
x)\left((k+n x)^2+3 n^2 y^2\right)}{n^3\left((k+n x)^2+n^2
y^2\right)^2} \\
=& - \frac{ x(\pi^2 k^2-9y^2-3x^2) }{ 3k^4 } \nonumber \\
& - \pi \sin\left(\frac{2 \pi k x}{x^2+y^2}\right) \frac{(x^2+y^2) \left( \cosh^2 \frac{\pi k y}{x^2+y^2} - \cos^2 \frac{\pi k x}{x^2+y^2} \right) + k \pi y \sinh \frac{2 k \pi y}{x^2+y^2}}{2k^3 \left( \cosh^2 \frac{\pi k y}{x^2+y^2} - \cos^2 \frac{\pi k x}{x^2+y^2} \right)^2}. \nonumber
\end{align*}
If $2x/(x^2+y^2)\in\mathbb{Z}$, then the second term vanishes.  Thus we are left with
\begin{align*}
T(x,y)=&\frac{\pi^4}{45}\frac{x (x^2 + 3 y^2)}{(x^2 + y^2)^2}-\sum_{\substack{k\\k\ne0}}\frac{x(\pi^2 k^2-9y^2-3x^2) }{ 3k^4}\\
=&\frac{\pi^4}{45}x\left(\frac{x^2+3 y^2}{\left(x^2+y^2\right)^2}+x^2+3 y^2-5\right),
\end{align*}
and \eqref{final F(q) integer case} follows.
\end{proof}

{\allowdisplaybreaks
\begin{table}
    \begin{tabular}{c c p{6 in}}
        \hline
        \\
        $q$  & $F(q)$\\
\\
        \hline\hline
        \\
        $e^{-2\pi}$ &  $80L_{-4}(2)$\\
        $e^{-2\pi\sqrt{2}}$ & $80\sqrt{2}L_{-8}(2)$ \\
        $e^{-2\pi\sqrt{3}}$ & $135\sqrt{3}L_{-3}(2)$ \\
        $e^{-2\pi\sqrt{4}}$ & $280L_{-4}(2)$ \\
        $e^{-2\pi\sqrt{5}}$ & $100 \sqrt{5} L_{-20}(2)+96 L_{-4}(2)$ \\
        $e^{-2\pi\sqrt{6}}$ & $120 \sqrt{6} L_{-24}(2)+90 \sqrt{3} L_{-3}(2)$ \\
        $e^{-2\pi\sqrt{7}}$ & $175\sqrt{7}L_{-7}(2)$\\
        $e^{-2\pi\sqrt{8}}$ & $280 \sqrt{2}L_{-8}(2)+240L_{-4}(2)$\\
        $e^{-2\pi\sqrt{9}}$ & $560L_{-4}(2)+180 \sqrt{3} L_{-3}(2)$\\
        $e^{-2\pi\sqrt{10}}$ & $200 \sqrt{10}L_{-40}(2)+ 192 \sqrt{2}L_{-8}(2) $\\
        $e^{-2\pi\sqrt{12}}$ & $480 L_{-4}(2)+\frac{1035}{2} \sqrt{3} L_{-3}(2) $\\
        $e^{-2\pi\sqrt{13}}$ & $260 \sqrt{13}L_{-52}(2)+480 L_{-4}(2)$\\
        $e^{-2\pi\sqrt{15}}$ & $\frac{375}{2} \sqrt{15} L_{-15}(2) + 468 \sqrt{3} L_{-3}(2)$\\
        $e^{-2\pi\sqrt{16}}$ & $480 \sqrt{2} L_{-8}(2) + 1100 L_{-4}(2)$\\
        $e^{-2\pi\sqrt{18}}$ & $880 \sqrt{2} L_{-8}(2)+540 \sqrt{3} L_{-3}(2)$\\
        $e^{-2\pi\sqrt{21}}$ & $210 \sqrt{21} L_{-84}(2)+210 \sqrt{7} L_{-7}(2)+480 L_{-4}(2)+360 \sqrt{3} L_{-3}(2)$\\
        $e^{-2\pi\sqrt{22}}$ & $440 \sqrt{22} L_{-88}(2) + 330 \sqrt{11}L_{-11}(2)$\\
        $e^{-2\pi\sqrt{24}}$ & $420 \sqrt{6}L_{-24}(2)+480 \sqrt{2}L_{-8}(2)+720 L_{-4}(2)+495 \sqrt{3} L_{-3}(2)$\\
        $e^{-2\pi\sqrt{25}}$ & $480 \sqrt{5} L_{-20}(2) + 2320 L_{-4}(2)$\\
        $e^{-2\pi\sqrt{28}}$ & $\frac{1435}{2} \sqrt{7} L_{-7}(2)+1920 L_{-4}(2)$\\
        $e^{-2\pi\sqrt{30}}$ & $300 \sqrt{30}L_{-120}(2)+288 \sqrt{6}L_{-24}(2)+225 \sqrt{15}L_{-15}(2)+630 \sqrt{3}L_{-3}(2)$\\
        $e^{-2\pi\sqrt{33}}$ & $330 \sqrt{33} L_{-132}(2)+330 \sqrt{11} L_{-11}(2)+1440 L_{-4}(2)+630 \sqrt{3}
        L_{-3}(2)$\\
        \hline\\
    \end{tabular}
    \caption{Select values of $F(q)$}
    \label{Table:F(q) formulas}
\end{table}
}

\subsection{Convergent rational formulas}
Now we prove rational, convergent, companion series formulas.  Virtually all of these results have appeared in the literature before, although we believe this is the first unified treatment of all of the formulas.  Equation \eqref{1/2 case formula 2} was proved by Zeilberger \cite[Theorem~8]{ZB}.  Formulas \eqref{1/2 case formula 0}, \eqref{1/2 case formula 1}, \eqref{1/2 case formula 3} are due to Guillera \cite{GuilleraThesis}, \cite{G10Rama}.  Equations \eqref{1/3 case formula 1} through \eqref{1/4 case formula 2} were conjectured by Sun using numerical experiments \cite{Sun}.  Formula \eqref{1/4 case formula 1} was subsequently proved by Guillera \cite{G}, and the Hessami-Pilehroods proved \eqref{1/4 case formula 2} \cite{HP}.  Our strategy is to express each
companion series in terms of $F(q)$'s, and then to evaluate $F(q)$
using properties of Epstein zeta functions. The hypergeometric-side
of the formula also requires values of $(a, b, z)$.  It is
straight-forward, albeit tedious, to calculate those quantities.  We
summarize the values of $(a,b,z)$ and the corresponding $q$'s in
Table \ref{Table: Sun's formulas a,b,z,q}.
\begin{theorem}\label{Theorem: proving Sun's formulas} The following formulas are true:
{\allowdisplaybreaks
\begin{align}
&\sum_{n=1}^{\infty}(-1)^{n+1} \frac{(1)_n^3}{\left(\frac{1}{2} \right)_n^3} \frac{(4n-1)}{n^3}=16L_{-4}(2),\label{1/2 case formula 0}\\
%
&\sum_{n=1}^{\infty} \frac{(1)_n^3}{\left(\frac{1}{2} \right)_n^3} \frac{(3n-1)}{n^3}\frac{1}{2^{2n}} =\frac{\pi^2}{2},\label{1/2 case formula 1}\\
&\sum_{n=1}^{\infty} \frac{(1)_n^3}{\left(\frac{1}{2} \right)_n^3} \frac{(21n-8)}{n^3}\frac{1}{2^{6n}} =\frac{\pi^2}{6},\label{1/2 case formula 2}\\
%
&\sum_{n=1}^{\infty} (-1)^{n+1}\frac{(1)_n^3}{\left(\frac{1}{2} \right)_n^3} \frac{(3n-1)}{n^3}\frac{1}{2^{3n}} =2L_{-4}(2),\label{1/2 case formula 3}\\
&\sum_{n=1}^{\infty} \frac{(1)_n^3}{\left(\frac{1}{2} \right)_n \left( \frac{1}{3} \right)_n \left(\frac{2}{3} \right)_n} \frac{(10n-3)}{n^3}\left( \frac{2}{27} \right)^{2n} =\frac{\pi^2}{2},\label{1/3 case formula 1}\\
%
&\sum_{n=1}^{\infty} \frac{(1)_n^3}{\left(\frac{1}{2} \right)_n \left( \frac{1}{3} \right)_n \left(\frac{2}{3} \right)_n} \frac{(11n-3)}{n^3}\left( \frac{16}{27} \right)^n =8 \pi^2,\label{1/3 case formula 2}\\
%
&\sum_{n=1}^{\infty}(-1)^{n+1}  \frac{(1)_n^3}{\left(\frac{1}{2} \right)_n \left( \frac{1}{3} \right)_n \left(\frac{2}{3} \right)_n} \frac{(15n-4)}{n^3}\frac{1}{4^n} =27 L_{-3}(2),\label{1/3 case formula 3}\\
%
&\sum_{n=1}^{\infty} (-1)^{n+1}\frac{(1)_n^3}{\left(\frac{1}{2} \right)_n \left( \frac{1}{4} \right)_n \left(\frac{3}{4} \right)_n} \frac{(5n-1)}{n^3}\left( \frac{3}{4} \right)^{2n} =\frac{45}{2} L_{-3}(2),\label{1/4 case formula 1}\\
%
&\sum_{n=1}^{\infty} \frac{(1)_n^3}{\left(\frac{1}{2} \right)_n \left( \frac{1}{4} \right)_n \left(\frac{3}{4} \right)_n} \frac{(35n-8)}{n^3}\left( \frac{3}{4} \right)^{4n} =12 \pi^2\label{1/4 case formula 2}.
\end{align}}
\end{theorem}
\begin{proof}
We begin by proving \eqref{1/2 case formula 0}. Set
$q=-e^{-\pi\sqrt{2}}$ in \eqref{5F4 1/2 case in terms of F}. We have
$(a,b,z)=\left(\frac{1}{2},2,-1\right)$.  The formula reduces to
\begin{equation*}
\frac{1}{2}\sum_{n=1}^{\infty}(-1)^{n+1}
\frac{(1)_n^3}{\left(\frac{1}{2} \right)_n^3}
\frac{(4n-1)}{n^3}=-\frac{1}{15}F\left(-e^{-\pi\sqrt{2}}\right)+\frac{1}{60}F\left(e^{-4\pi\sqrt{2}}\right).
\end{equation*}
Apply \eqref{final F(q) evaluation} to reduce the equation to
\begin{align*}
\sum_{n=1}^{\infty}(-1)^{n+1}\frac{(1)_n^3}{\left(\frac{1}{2}
\right)_n^3} \frac{(4n-1)}{n^3}=&\frac{64 \sqrt{2}}{\pi ^2}
S(1,0,8;2)-\frac{4 \sqrt{2}}{\pi ^2}
S\left(1,1,\frac{3}{4};2\right)\\
=&\frac{64 \sqrt{2}}{\pi ^2}\left(
S(1,0,8;2)-S\left(3,4,4;2\right)\right).
\end{align*}
Glasser and Zucker have evaluated $S(1,0,8;t)$ for all $t$
\cite{GZ}. Their method also applies to $S(3,4,4;t)=S(3,2,3;t)$.  When $t=2$,
the formulas become
\begin{align*}
S(1,0,8;2)=&\frac{7\pi^2}{48} L_{-8}(2)+\frac{\pi
^2}{8\sqrt{2}}L_{-4}(2),\\
S(3,4,4;2)=&\frac{7\pi^2}{48} L_{-8}(2)-\frac{\pi
^2}{8\sqrt{2}}L_{-4}(2),
\end{align*}
and the result follows.

Next consider \eqref{1/2 case formula 1}. Set
$q=ie^{-\pi\sqrt{3}/2}$ in \eqref{5F4 1/2 case in terms of F}. We
have $(a,b,z)=\left(-\frac{i}{2},-\frac{3i}{2},4\right)$.  The
formula reduces to
\begin{equation*}
\frac{i}{2}\sum_{n=1}^{\infty} \frac{(1)_n^3}{\left(\frac{1}{2}
\right)_n^3}
\frac{(3n-1)}{n^3}\frac{1}{2^{2n}}=\frac{3i\pi^2}{8}-\frac{1}{15}F\left(ie^{-\pi\sqrt{3}/2}\right)+\frac{1}{60}F\left(e^{-2\pi\sqrt{3}}\right).
\end{equation*}
Equate the imaginary parts, and apply \eqref{final F(q) integer case}. The equation reduces to
\begin{align*}
\sum_{n=1}^{\infty} \frac{(1)_n^3}{\left(\frac{1}{2} \right)_n^3} \frac{(3n-1)}{n^3}\frac{1}{2^{2n}}=&\frac{3\pi^2}{4}-\frac{2}{15}\Im F\left(ie^{-\pi\sqrt{3}/2}\right)\\
=&\frac{\pi^2}{2}.
\end{align*}

Next we prove \eqref{1/2 case formula 2}.  Set $q=e^{3\pi
i/4}e^{-\pi\sqrt{7}/4}$ in \eqref{5F4 1/2 case in terms of F}.  We
have $(a,b,z)=\left(-2i,-\frac{21i}{4},64\right)$.  The formula
reduces to
\begin{align*}
\frac{i}{4}\sum_{n=1}^{\infty} \frac{(1)_n^3}{\left(\frac{1}{2} \right)_n^3} \frac{(21n-8)}{n^3}\frac{1}{2^{6n}}=\frac{9\pi^2 i}{64}-\frac{1}{15}F\left(e^{3\pi i/4}e^{-\pi\sqrt{7}/4}\right)+\frac{1}{60}F\left(-e^{-\pi\sqrt{7}}\right).
\end{align*}
Equate the imaginary parts, then apply \eqref{final F(q) integer case}.  We obtain
\begin{align*}
\sum_{n=1}^{\infty} \frac{(1)_n^3}{\left(\frac{1}{2} \right)_n^3} \frac{(21n-8)}{n^3}\frac{1}{2^{6n}}=&\frac{9\pi^2 }{16}-\frac{4}{15}\Im F\left(e^{3\pi i/4}e^{-\pi\sqrt{7}/4}\right)\\
=&\frac{\pi^2}{6}.
\end{align*}

Next consider \eqref{1/2 case formula 3}.  Set $q=-e^{-\pi}$ in
\eqref{5F4 1/2 case in terms of F}.  We have $(a,b,z)=(1,3,-8)$. The
formula reduces to
\begin{align*}
\sum_{n=1}^{\infty} \frac{(1)_n^3}{\left(\frac{1}{2} \right)_n^3} \frac{(3n-1)}{n^3}\frac{(-1)^{n+1}}{2^{3n}}=-\frac{1}{15}F\left(-e^{-\pi}\right)+\frac{1}{60}F\left(e^{-4\pi}\right).
\end{align*}
Apply \eqref{final F(q) evaluation} to
obtain
\begin{align*}
\sum_{n=1}^{\infty} \frac{(1)_n^3}{\left(\frac{1}{2} \right)_n^3}
\frac{(3n-1)}{n^3}\frac{(-1)^{n+1}}{2^{3n}}=&-\frac{1}{\pi^2}S\left(1,1,\frac{1}{2};2\right)+\frac{16}{\pi^2}S(1,0,4;2)\\
=&2L_{-4}(2).
\end{align*}
In the final step we used $S(1,0,4;2)=\frac{7\pi^2}{24}
L_{-4}(2)$, and
$S\left(1,1,\frac{1}{2};2\right)=4S(2,2,1;2)=4S(1,0,1;2)=\frac{8\pi^2}{3}L_{-4}(2)$.  Both of these evaluations follow from the results of
Glasser and Zucker \cite{GZ}.

Now consider \eqref{1/3 case formula 1}.  Set $q=e^{2\pi
i/3}e^{-2\pi\sqrt{2}/3}$ in \eqref{5F4 1/3 case in terms of F}. We
have $(a,b,z)=\left(-i,-\frac{10i}{3},\frac{27}{2}\right)$.  The
formula reduces to
\begin{align*}
\frac{i}{3}\sum_{n=1}^{\infty} \frac{(1)_n^3}{\left(\frac{1}{3} \right)_n\left(\frac{1}{2}\right)_n\left(\frac{2}{3}\right)_n} \frac{(10n-3)}{n^3}\left(\frac{2}{27}\right)^n=\frac{26\pi^2 i}{81}-\frac{1}{8}F\left(e^{2\pi i/3}e^{-2\pi\sqrt{2}/3}\right)+\frac{1}{24}F\left(e^{-2\pi\sqrt{2}}\right).
\end{align*}
Take imaginary parts, then apply \eqref{final F(q) integer case}. We obtain
\begin{align*}
\sum_{n=1}^{\infty} \frac{(1)_n^3}{\left(\frac{1}{3} \right)_n\left(\frac{1}{2}\right)_n\left(\frac{2}{3}\right)_n} \frac{(10n-3)}{n^3}\left(\frac{2}{27}\right)^n&=\frac{26\pi^2}{27}-\frac{3}{8}\Im F\left(e^{2\pi i/3}e^{-2\pi\sqrt{2}/3}\right)\\
&=\frac{\pi^2}{2}.
\end{align*}

Next we prove \eqref{1/3 case formula 2}.  Set $q=e^{\pi
i/3}e^{-\pi\sqrt{11}/3}$ in \eqref{5F4 1/3 case in terms of F}. We
have
$(a,b,z)=\left(-\frac{i}{4},-\frac{11i}{12},\frac{27}{16}\right)$.
The formula reduces to
\begin{align*}
\frac{i}{12}\sum_{n=1}^{\infty} \frac{(1)_n^3}{\left(\frac{1}{3} \right)_n\left(\frac{1}{2}\right)_n\left(\frac{2}{3}\right)_n} \frac{(11n-3)}{n^3}\left(\frac{16}{27}\right)^n=\frac{64\pi^2 i}{81}-\frac{1}{8}F\left(e^{\pi i/3}e^{-\pi\sqrt{11}/3}\right)+\frac{1}{24}F\left(-e^{-\pi\sqrt{11}}\right).
\end{align*}
Take imaginary parts, then apply \eqref{final F(q) integer case}.  We have
\begin{align*}
\sum_{n=1}^{\infty} \frac{(1)_n^3}{\left(\frac{1}{3} \right)_n\left(\frac{1}{2}\right)_n\left(\frac{2}{3}\right)_n} \frac{(11n-3)}{n^3}\left(\frac{16}{27}\right)^n&=\frac{256\pi^2}{27}-\frac{3}{2}\Im F\left(e^{\pi i/3}e^{-\pi\sqrt{11}/3}\right)\\
&=8\pi^2.
\end{align*}

Now prove \eqref{1/3 case formula 3}.  Set $q=-e^{-\pi\sqrt{15}/3}$
in \eqref{5F4 1/3 case in terms of F}. We have
$(a,b,z)=\left(\frac{4}{3\sqrt{3}},\frac{5}{\sqrt{3}},-4\right)$.
The formula reduces to
\begin{align*}
\frac{1}{3\sqrt{3}}\sum_{n=1}^{\infty} \frac{(1)_n^3}{\left(\frac{1}{3} \right)_n\left(\frac{1}{2}\right)_n\left(\frac{2}{3}\right)_n} \frac{(15n-4)}{n^3}\frac{(-1)^{n+1}}{4^n}=-\frac{1}{8}F\left(-e^{-\pi\sqrt{15}/3}\right)+\frac{1}{24}F\left(-e^{-\pi\sqrt{15}}\right).
\end{align*}
Apply \eqref{final F(q) evaluation} to
obtain
\begin{align*}
\sum_{n=1}^{\infty} \frac{(1)_n^3}{\left(\frac{1}{3} \right)_n\left(\frac{1}{2}\right)_n\left(\frac{2}{3}\right)_n} \frac{(15n-4)}{n^3}\frac{(-1)^{n+1}}{4^n}=&-\frac{75\sqrt{5}}{8\pi^2}S\left(1,1,\frac{2}{3};2\right)+\frac{675\sqrt{5}}{8\pi^2}S\left(1,1,4;2\right)\\
=&\frac{675\sqrt{5}}{8\pi^2}\left(S(1,1,4;2)-S(2,3,3;2)\right).
\end{align*}
Glasser and Zucker have calculated $S(1,1,4;t)$ for all $t$ \cite{GZ}. Their method also applies
to $S(2,3,3;t)=S(2,1,2;t)$.  When $t=2$ the formulas reduce to
\begin{align*}
S(1,1,4;2)=&\frac{\pi^2}{6}L_{-15}(2)+\frac{4\pi^2}{25\sqrt{5}}L_{-3}(2),\\
S(2,3,3;2)=&\frac{\pi^2}{6}L_{-15}(2)-\frac{4\pi^2}{25\sqrt{5}}L_{-3}(2),
\end{align*}
and \eqref{1/3 case formula 3} follows.

Next we prove \eqref{1/4 case formula 1}.  Set $q=-e^{-\pi\sqrt{3}}$
in \eqref{5F4 1/4 case in terms of F}. We have
$(a,b,z)=\left(\frac{1}{\sqrt{3}},\frac{5}{\sqrt{3}},-\frac{16}{9}\right)$.
The formula reduces to
\begin{align*}
\frac{1}{\sqrt{3}}\sum_{n=1}^{\infty} \frac{(1)_n^3}{\left(\frac{1}{4} \right)_n\left(\frac{1}{2}\right)_n\left(\frac{3}{4}\right)_n} \frac{(5n-1)}{n^3}(-1)^{n+1}\left(\frac{3}{4}\right)^{2n}=-\frac{1}{3}F\left(-e^{-\pi\sqrt{3}}\right)+\frac{1}{6}F\left(e^{-2\pi\sqrt{3}}\right).
\end{align*}
By \eqref{final F(q) evaluation}, we have
\begin{align*}
\sum_{n=1}^{\infty} \frac{(1)_n^3}{\left(\frac{1}{4} \right)_n\left(\frac{1}{2}\right)_n\left(\frac{3}{4}\right)_n} \frac{(5n-1)}{n^3}(-1)^{n+1}\left(\frac{3}{4}\right)^{2n}=&-\frac{45}{\pi^2}S(1,1,1;2)+\frac{180}{\pi^2}S(1,0,3;2)\\
=&\frac{45}{2}L_{-3}(2).
\end{align*}
Glasser and Zucker proved that $S(1,0,3;2)=\frac{3\pi^2}{8}L_{-3}(2)$, and $S(1,1,1;2)=\pi^2 L_{-3}(2)$ \cite{GZ}.

Finally prove \eqref{1/4 case formula 2}.  Set $q=i
e^{-\pi\sqrt{7}/2}$ in \eqref{5F4 1/4 case in terms of F}. We have
$(a,b,z)=\left(-\frac{4 i}{9},-\frac{35
i}{18},\frac{256}{81}\right)$.  The formula reduces to
\begin{align*}
\frac{i}{18}\sum_{n=1}^{\infty} \frac{(1)_n^3}{\left(\frac{1}{4} \right)_n\left(\frac{1}{2}\right)_n\left(\frac{3}{4}\right)_n} \frac{(35n-8)}{n^3}\left(\frac{3}{4}\right)^{4n}=\frac{7\pi^2i}{8}-\frac{1}{3}F\left(i e^{-\pi\sqrt{7}/2}\right)+\frac{1}{6}F\left(-e^{-\pi\sqrt{7}}\right).
\end{align*}
Take the imaginary part, then apply \eqref{final F(q) integer case}.  We obtain
\begin{align*}
\sum_{n=1}^{\infty} \frac{(1)_n^3}{\left(\frac{1}{4} \right)_n\left(\frac{1}{2}\right)_n\left(\frac{3}{4}\right)_n} \frac{(35n-8)}{n^3}\left(\frac{3}{4}\right)^{4n}=&\frac{63\pi^2}{4}-6\Im F\left(i e^{-\pi\sqrt{7}/2}\right)\\
=&12\pi^2.
\end{align*}
\end{proof}

\begin{table}
    \begin{tabular}{c c c c c p{6 in}|}
        \hline
        \\
        $s$ & $q$ & $a$ & $b$ & $z$  \\
        \\
        \hline\hline
        \\
        $\frac{1}{2}$ & $-e^{-\pi\sqrt{2}}$ & $\frac{1}{2}$ & $2$ &
        $-1$\\\\
        $\frac{1}{2}$ & $ie^{-\pi\sqrt{3}/2}$ & $-\frac{i}{2}$ & $-\frac{3i}{2}$ &
        $4$\\\\
        $\frac{1}{2}$ & $e^{3\pi i/4}e^{-\pi\sqrt{7}/4}$ & $-2i$ & $-\frac{21i}{4}$ &
        $64$\\\\
        $\frac{1}{2}$ & $-e^{-\pi}$ & $1$ & $3$ &
        $-8$\\\\
        $\frac{1}{3}$ & $e^{2\pi i/3}e^{-2\pi\sqrt{2}/3}$ & $-i$ & $-\frac{10i}{3}$ &
        $\frac{27}{2}$\\\\
        $\frac{1}{3}$ & $e^{\pi i/3}e^{-\pi\sqrt{11}/3}$ & $-\frac{i}{4}$ & $-\frac{11i}{12}$ &
        $\frac{27}{16}$\\\\
        $\frac{1}{3}$ & $-e^{-\pi\sqrt{15}/3}$ & $\frac{4}{3\sqrt{3}}$ & $\frac{5}{\sqrt{3}}$ &
        $-4$\\\\
        $\frac{1}{4}$ & $-e^{-\pi\sqrt{3}}$ & $\frac{1}{\sqrt{3}}$ & $\frac{5}{\sqrt{3}}$ &
        $-\frac{16}{9}$\\\\
        $\frac{1}{4}$& $ie^{-\pi\sqrt{7}/2}$ & $-\frac{4i}{9}$ & $-\frac{35i}{18}$& $\frac{256}{81}$\\
        \\
        \hline\\
    \end{tabular}
\caption{Values of $(a,b,z)$ in Theorem \ref{Theorem: proving Sun's
formulas}}
\label{Table: Sun's formulas a,b,z,q}
\end{table}

Table \ref{Table: Sun's formulas a,b,z,q}
summarizes the values of
$(a, b, z)$ and $q$ in Theorem \ref{Theorem:
proving Sun's formulas}.  These values also lead to divergent formulas
for $1/\pi$. For instance, when
$s=\frac13$ and $(a,b,z)=\left(\frac{4}{3\sqrt{3}},\frac{5}{\sqrt{3}},-4\right)$,
we obtain \eqref{1/3 case formula 3}, and
\begin{equation*}
\frac{1}{\pi}=\frac{4}{3\sqrt{3}}
\hpg43{\frac{1}{3},\,\frac{1}{2},\,\frac{2}{3},\,\frac{19}{15}}{1,1,\frac{4}{15}}{\,-4}.
\end{equation*}
The
right-hand side equals $.3183098\dots$, which
agrees perfectly with the expected numerical value of $1/\pi$.

\subsection{Divergent rational formulas}
Next we examine divergent hypergeometric formulas for Dirichlet $L$-values.  These are companions to the convergent formulas for $1/\pi$.  Since the identities have $|z|<1$, we have substituted a ${_5{\ff}_4}$ function
for the divergent companion series:
\begin{equation}\label{Companion in terms of 5F4}
\begin{split}
\sum_{n=1}^{\infty}&\frac{(1)_n^3}{(s)_n\left(\frac{1}{2}\right)_n(1-s)_n}\frac{(a-b n)}{n^3}z^{-n}\\
&=
\frac{2(a-b)}{s(1-s)z}\hpg54{1,\,1,\,1,\,1,\,2-\frac{a}{b}}{\frac32,1+s,2-s,1-\frac{a}{b}}{\,z^{-1}}.
\end{split}
\end{equation}
The ${_5{\ff}_4}$ function has a branch cut on the interval $[1,\infty)$ \cite{DL}.  When $z^{-1}$ lies on the branch cut, the function takes a complex value.  The real part of the function is uniquely defined, but the sign of the imaginary part depends on how we approach the branch cut.  We use the same computational method as \texttt{Mathematica 8}, namely when $z^{-1}\in[1,\infty)$, we define ${_5{\ff}_4}\left(\dots\bigm|z^{-1}\right)=\lim_{\delta\mapsto0}{{_5{\ff}_4}}\left(\dots\bigm|z^{-1}-i\delta\right)$.

\begin{theorem}\label{Theorem: Divergent rational cases}
The following identity holds:
\begin{equation}
\frac{2(a-b)}{s(1-s)z}\hpg54{1,\,1,\,1,\,1,\,2-\frac{a}{b}}{\frac32,1+s,2-s,1-\frac{a}{b}}{\,z^{-1}}=L(2),
\end{equation}
for the values of $s$, $(a,b,z)$, and $L(2)$ in Tables \ref{Table:divergent abzq cases z<0} and \ref{Table:divergent abzq cases z>0}.
\end{theorem}
\begin{proof}  Proofs are the same as in Theorem \ref{Theorem: proving Sun's formulas}, so we only consider one example in detail.
Set $q=e^{-\pi \sqrt{7}}$ in \eqref{5F4 1/2 case in terms of F}.  By Table \ref{Table:divergent abzq cases z>0}, we have $s=\frac12$ and $(a,b,z)=\left(\frac{5}{16},\frac{21}{8},\frac{1}{64}\right)$.  Applying \eqref{final F(q) evaluation} and then \eqref{Companion in terms of 5F4}, reduces the formula reduces to
\begin{align*}
-1184~\hpg54{1,\,1,\,1,\,1,\,\frac{79}{42}}{\frac32,\frac32,\frac32,\frac{37}{42}}{\,64}
&=4 i \pi^2  -
 \frac{1}{15} F\left(e^{-\pi\sqrt{7}}\right)+ \frac{1}{60} F\left(e^{-4 \pi\sqrt{7}}\right)\\
 &=4 i \pi^2-\frac{112 \sqrt{7}}{\pi ^2}\left(S(4,0,7;2)-S(1,0,28;2)\right).
\end{align*}
By the results of Glasser and Zucker \cite{GZ}, we obtain
\begin{align*}
S(1,0,28;2)&=\frac{41\pi^2}{384}L_{-7}(2)+\frac{2 \pi ^2}{7 \sqrt{7}}L_{-4}(2),\\
S(4,0,7;2)&=\frac{41\pi^2}{384}L_{-7}(2)-\frac{2 \pi ^2}{7 \sqrt{7}}L_{-4}(2),
\end{align*}
and we recover the value of $L(2)$ in Table \ref{Table:divergent abzq cases z>0}.  After simplifying, we find that
\begin{equation*}
\hpg54{1,\,1,\,1,\,1,\,\frac{79}{42}}{\frac32,\frac32,\frac32,\frac{37}{42}}{\,64}
=-\frac{2}{37}L_{-4}(2)-\frac{1}{296}\pi^2i.
\end{equation*}
All of the formulas in Tables \ref{Table:divergent abzq cases z<0} and \ref{Table:divergent abzq cases z>0} follow from analogous arguments.
\end{proof}

\begin{table}
    \begin{tabular}{c c c c c c}
        \hline
        \\
        $s$ & $q$ & $a$ & $b$ & $z<0$ & $L(2)$ \\ \\
        \hline \hline \\
        $\frac12$ & $-e^{-\pi \sqrt{2}}$  & $\frac{1}{2}$ & $\frac42$ & $-1$ & $8L_{-4}(2)$ \\ \\
        $\frac12$ & $-e^{-\pi \sqrt{4}}$  & $\frac{1}{2\sqrt{2}}$ & $\frac{6}{2\sqrt2}$ & $-\frac{1}{8}$ & $16 \sqrt2 L_{-8}(2)$ \\ \\
        \hline \\
        $\frac13$ & $-e^{-\pi \sqrt{9/3}}$  & $\frac{\sqrt{3}}{4}$ & $\frac{5\sqrt{3}}{4}$ & $-\frac{9}{16}$ & $10 \sqrt3 L_{-3}(2)$ \\ \\
        $\frac13$ & $-e^{-\pi \sqrt{17/3}}$ & $\frac{7}{12\sqrt3}$ & $\frac{51}{12\sqrt3}$ & $-\frac{1}{16}$ & $30 \sqrt3 L_{-3}(2)$ \\ \\
        $\frac13$ & $-e^{-\pi \sqrt{25/3}}$ & $\frac{\sqrt{15}}{12}$ & $\frac{9\sqrt{15}}{12}$ & $-\frac{1}{80}$ & $15\sqrt{15} L_{-15}(2)$ \\ \\
        $\frac13$ & $-e^{-\pi \sqrt{41/3}}$ & $\frac{106}{192\sqrt3}$ & $\frac{1230}{192\sqrt3}$ & $-\frac{1}{2^{10}}$ & $120 \sqrt3 L_{-3}(2)$ \\ \\
        $\frac13$ & $-e^{-\pi \sqrt{49/3}}$ & $\frac{26\sqrt7}{216}$ & $\frac{330\sqrt7}{216}$ & $-\frac{1}{3024}$ & $70\sqrt7 L_{-7}(2)$ \\ \\
        $\frac13$ & $-e^{-\pi \sqrt{89/3}}$ & $\frac{827}{1500\sqrt3}$ & $\frac{14151}{1500\sqrt3}$ & $-\frac{1}{500^2}$ & $390 \sqrt3 L_{-3}(2)$ \\ \\
        \hline \\
        $\frac14$ & $-e^{-\pi \sqrt{5}}$  & $\frac{3}{8}$ & $\frac{20}{8}$ & $-\frac{1}{4}$ & $32L_{-4}(2)$ \\ \\
        $\frac14$ & $-e^{-\pi \sqrt{7}}$  & $\frac{8}{9\sqrt7}$ & $\frac{65}{9\sqrt7}$ & $-\frac{16^2}{63^2}$ & $\frac{35}{2} \sqrt7 L_{-7}(2)$ \\ \\
        $\frac14$ & $-e^{-\pi \sqrt{9}}$  & $\frac{3\sqrt3}{16}$ & $\frac{28\sqrt3}{16}$ & $-\frac{1}{48}$ & $60 \sqrt3 L_{-3}(2)$ \\ \\
        $\frac14$ & $-e^{-\pi \sqrt{13}}$ & $\frac{23}{72}$ & $\frac{260}{72}$ & $-\frac{1}{18^2}$ & $160 L_{-4}(2)$ \\ \\
        $\frac14$ & $-e^{-\pi \sqrt{25}}$ & $\frac{41\sqrt5}{288}$ & $\frac{644\sqrt5}{288}$ & $-\frac{1}{5 \cdot 72^2}$ & $160 \sqrt5 L_{-20}(2)$ \\ \\
        $\frac14$ & $-e^{-\pi \sqrt{37}}$ & $\frac{1123}{3528}$ & $\frac{21460}{3528}$ & $-\frac{1}{882^2}$ & $800 L_{-4}(2)$ \\ \\
        \hline
    \end{tabular}
    \vskip 0.5cm
\caption{Values of $(a,b,z)$ with $z<0$ in Theorem \ref{Theorem: Divergent rational cases}}
\label{Table:divergent abzq cases z<0}
\end{table}

\begin{table}
    \begin{tabular}{c c c c c c}
        \hline
        \\
        $s$ & $q$ & $a$ & $b$ & $z>0$ & $L(2)$ \\ \\
        \hline \hline \\
        $\frac12$ & $e^{-\pi \sqrt{3}}$ & $\frac{1}{4}$ & $\frac{6}{4}$ & $\frac{1}{4}$ & $16L_{-4}(2)+2\pi^2 i$ \\ \\
        $\frac12$ & $e^{-\pi \sqrt{7}}$ & $\frac{5}{16}$ & $\frac{42}{16}$ & $\frac{1}{64}$ & $64L_{-4}(2)+4\pi^2 i$ \\ \\
        \hline \\
        $\frac13$ & $e^{-\pi \sqrt{8/3}}$  & $\frac{1}{3\sqrt{3}}$ & $\frac{6}{3\sqrt{3}}$ & $\frac{1}{2}$ & $\frac{15}{2}\sqrt3 L_{-3}(2)+2\pi^2 i$ \\ \\
        $\frac13$ & $e^{-\pi \sqrt{16/3}}$ & $\frac{8}{27}$ & $\frac{60}{27}$ & $\frac{2}{27}$ & $40L_{-4}(2)+\frac{10}{3}\pi^2 i$ \\ \\
        $\frac13$ & $e^{-\pi \sqrt{20/3}}$  & $\frac{8}{15\sqrt3}$ & $\frac{66}{15\sqrt3}$ & $\frac{4}{125}$ & $39\sqrt3 L_{-3}(2)+4\pi^2 i$ \\ \\
        \hline \\
        $\frac14$ & $e^{-2\pi}$  & $\frac{2}{9}$ & $\frac{14}{9}$ & $\frac{32}{81}$ & $20L_{-4}(2)+3\pi^2 i$ \\ \\
        $\frac14$ & $e^{-\pi \sqrt{6}}$  & $\frac{1}{2\sqrt3}$ & $\frac{8}{2\sqrt3}$ & $\frac{1}{9}$ & $30 \sqrt3 L_{-3}(2)+4\pi^2 i$ \\ \\
        $\frac14$ & $e^{-\pi \sqrt{10}}$ & $\frac{4}{9\sqrt2}$ & $\frac{40}{9\sqrt2}$ & $\frac{1}{81}$ & $64 \sqrt2 L_{-8}(2)+6\pi^2 i$ \\ \\
        $\frac14$ & $e^{-\pi \sqrt{18}}$ & $\frac{27}{49\sqrt3}$ & $\frac{360}{49\sqrt3}$ & $\frac{1}{7^4}$ & $180 \sqrt3 L_{-3}(2)+10\pi^2 i$ \\ \\
        $\frac14$ & $e^{-\pi \sqrt{22}}$ & $\frac{19}{18\sqrt{11}}$ & $\frac{280}{18\sqrt{11}}$ & $\frac{1}{99^2}$ & $110\sqrt{11}L_{-11}(2)+12\pi^2 i$ \\ \\
        $\frac14$ & $e^{-\pi \sqrt{58}}$ & $\frac{4412}{9801\sqrt2}$ & $\frac{105560}{9801\sqrt2}$ & $\frac{1}{99^4}$ & $960\sqrt2 L_{-8}(2)+30\pi^2 i$ \\ \\
        \hline
    \end{tabular}
    \vskip 0.5cm
\caption{Values of $(a,b,z)$ with $z>0$ in Theorem \ref{Theorem: Divergent rational cases}}
\label{Table:divergent abzq cases z>0}
\end{table}

\subsection{Irrational formulas}

We emphasize that the \textit{vast majority} of companion series formulas involve irrational values of $(a,b,z)$.  Consider the narrow class of formulas
which arises from setting $q=e^{-2\pi\sqrt{v}}$ in \eqref{5F4 1/4 case in terms of F}.  The companion series with $s=\frac14$ reduces to a linear combination of $S(1,0,v;2)$, $S(1,0,4v;2)$, and elementary constants.  There are $24$ values of $v\in\mathbb{N}$, for which both sums reduces to Dirichlet $L$-values \cite{GZ}.  The $v=1$ case produces a rational, albeit divergent, companion series (Theorem \ref{Theorem: Divergent rational cases} with $s=\frac14$ and $(a,b,z)=\left(\frac29,\frac{14}{9},\frac{32}{81}\right)$).  The other $23$ choices lead to formulas with complicated algebraic values of $(a, b, z)$.  While it is possible to determine those numbers from modular equations, it is usually much easier to use a computer.  Formulas \eqref{dq-dz} and \eqref{ecu-a-b} are rather unwieldy for computational purposes, so we found it convenient to use theta functions.  Suppose that $s=\frac{1}{2}$, and that $q$ lies in a neighborhood of zero. Then
\begin{equation}\label{param s=1/2}
\begin{split}
z=&4\frac{\theta_3^4(-q)}{\theta_3^4(q)}\left(1-\frac{\theta_3^4(-q)}{\theta_3^4(q)}\right),\\
a=&\frac{1}{\pi\theta_3^4(q)}\left(1+\frac{8\log|q|}{\theta_3(q)}\sum_{n=1}^{\infty}n^2 q^{n^2}\right),\\
b=&\frac{\log|q|}{\pi}\left(1-2\frac{\theta_3^4(-q)}{\theta_3^4(q)}\right),
\end{split}
\end{equation}
where
\begin{equation*}
\theta_3(q)=1+2\sum_{n=1}^{\infty}q^{n^2}.
\end{equation*}
More complicated formulas are required if $s\in\{\frac13,\frac14\}$.

To give an example of an irrational formula, set $q=e^{9\pi i/8}e^{-\pi\sqrt{15}/8}$ in \eqref{5F4 1/2 case in terms of F}.  We calculate $(a,b,z)\approx(11.09i, 26.54i, 3006.63)$.  The PSLQ algorithm returns the following polynomials:
\begin{align*}
0=&1-11i a+a^2,\\
0=&495-1680 i b+64 b^2,\\
0=&4096-3008 z+z^2.
\end{align*}
Therefore $(a,b,z)=\left(\frac{1}{2} i \left(11+5 \sqrt{5}\right),\frac{3}{8} i \left(35+16 \sqrt{5}\right),\frac{1}{4}\left(1+\sqrt{5}\right)^8\right)$.  After simplifying with \eqref{final F(q) integer case}, we arrive at the following identity:
\begin{equation}\label{1/2 irrational pi^2 formula}
\frac{\pi^2}{30}= \sum_{n=1}^{\infty} \frac{3 (35 + 16 \sqrt{5})n-4 (11 + 5 \sqrt{5})}{n^3{2n\choose n}^3}\left(\frac{\sqrt{5}-1}{2}\right)^{8n}.
\end{equation}
This should be compared to Ramanujan's irrational formula for $1/\pi$, since both formulas involve powers of the golden ratio \cite{Ra}.
Table \ref{Table: Irrational evaluations} contains many additional irrational formulas.
\begin{table}
    \begin{tabular}{c c c c c c p{6 in}|}
        \hline
        \\
        $s$ & $q$ & $a$ & $b$ & $|z|>1$ & \text{Value of equation \eqref{Comp series}} \\
        \\
        \hline\hline
        \\
        $\frac{1}{2}$ & $-e^{-\pi\frac{\sqrt2}{2}}$ & $\frac{3+2\sqrt{2}}{2}$ & $\frac{8+5\sqrt{2}}{2}$ &
        $\frac{-8}{(\sqrt{2}-1)^3}$ & $2L_{-4}(2)-\sqrt{2} L_{-8}(2)$ \\\\
        $\frac{1}{2}$ & $-e^{-\frac{\pi}{2}}$ & $\frac{14+10\sqrt{2}}{2}$ & $\frac{33+24\sqrt{2}}{2}$ &
        $\frac{-16\sqrt2}{(\sqrt{2}-1)^6}$ & $-\frac{13}{4}L_{-4}(2)+2\sqrt{2} L_{-8}(2)$ \\\\
        $\frac{1}{2}$ & $-e^{-\pi\frac{\sqrt{2}}{3}}$ & $\frac{59+24\sqrt{6}}{6}$ & $\frac{140+56\sqrt{6}}{6}$ &
        $\frac{-1}{(5-2\sqrt{6})^4}$ & $\frac{136}{9}L_{-4}(2)-\frac{16}{3}\sqrt{6} L_{-24}(2)$ \\\\
        $\frac{1}{2}$ & $-e^{-\pi\frac{2\sqrt{3}}{3}}$ & $\frac{3\sqrt{6}+7\sqrt{2}}{24}$ & $\frac{6\sqrt{6}+30\sqrt{2}}{24}$ &
        $\frac{-1}{2(\sqrt{3}-1)^6}$ & $16\sqrt{2}L_{-8}(2)-8\sqrt{6} L_{-24}(2)$ \\\\
        $\frac{1}{2}$ & $-e^{-\pi\frac{\sqrt{6}}{3}}$ & $\frac{5+4\sqrt{2}}{6}$ & $\frac{12+12\sqrt{2}}{6}$ &
        $\frac{-1}{(\sqrt{2}-1)^4}$ & $-8L_{-4}(2)+\frac{16}{3}\sqrt{2} L_{-8}(2)$ \\\\
        $\frac{1}{2}$ & $-e^{-\pi \frac{\sqrt{10}}{5}}$ & $\frac{23+10\sqrt{5}}{10}$ & $\frac{60+24\sqrt{5}}{10}$ &
        $\frac{-1}{(\sqrt{5}-2)^4}$ & $\frac{56}{5}L_{-4}(2)-4\sqrt{5} L_{-20}(2)$ \\\\
        $\frac{1}{2}$ & $e^{\frac{9\pi i}{8}}e^{-\pi \frac{\sqrt{15}}{8}}$ & $\frac{4(11+5\sqrt{5})}{8}i$ & $\frac{3(35+16\sqrt5)}{8}i$ & $\frac{2^{14}}{(\sqrt{5}-1)^8}$ & $-\frac{1}{240}\pi^2 i$ \\\\
        $\frac{1}{3}$ & $-e^{-\pi \frac{\sqrt{21}}{3}}$ & $\frac{10+7\sqrt{7}}{54}$ & $\frac{21+39\sqrt{7}}{54}$ &
        $\frac{-1}{26\sqrt{7}-68}$ & $-20L_{-4}(2)+\frac{35}{4}\sqrt{7} L_{-7}(2)$ \\\\
        $\frac{1}{4}$ & $-e^{-\pi \frac{\sqrt{21}}{3}}$ & $\frac{27+20\sqrt3}{72}$ & $\frac{84+112\sqrt3}{72}$ & $\frac{-1}{(42-24\sqrt3)^2}$ & $-\frac{160}{3}L_{-4}(2)+40\sqrt{3}L_{-3}(2)$ \\\\
        $\frac{1}{4}$& $-e^{-\frac{3\pi\sqrt{5}}{5}}$ &  $\frac{3987+2124\sqrt3}{4840}$ & $\frac{19380+7440\sqrt3}{4840}$& $\frac{-1}{(680\sqrt3-1178)^2}$ & $\frac{544}{5}L_{-4}(2)-72\sqrt{3} L_{-3}(2)$ \\\\
        \hline\\
    \end{tabular}
\caption{Select convergent irrational companion series evaluations.}
\label{Table: Irrational evaluations}
\end{table}

\section{Irreducible values of $S(A,B,C;2)$}\label{sec: irreducible values of S}

Irreducible values of $S(A,B,C;2)$ occur when the quadratic form $A n^2+B n m+C m^2$ fails the one class per genus test.  Apart from a few oddball cases, it is probably impossible to reduce these sums to Dirichlet $L$-functions \cite{ZR}.  In this section, we prove that it is still possible to express some irreducible values of $S(A,B,C;2)$ in terms of hypergeometric functions.  Propositions \ref{Proposition:Companion series reduction} and \ref{proposition: F in terms of S} reduce every interesting companion series to two values of $S(A,B,C;2)$.  Sometimes it is possible to select $q$, so that one sum reduces to Dirichlet $L$-values, and one sum does not.  Sometimes both values of $S(A,B,C;2)$ are irreducible, but one of them can be eliminated by finding a multi-term linear dependence with Dirichlet $L$-functions.

To make a first attempt at finding a formula, set $q=e^{-3\pi}$ in \eqref{5F4 1/2 case in terms of F}.  Then
$s=\frac12$ and $(a,b,z)=\left(\frac{1}{4}(18 r - 5 r^3),12  r - 3r^3,(7+4\sqrt{3})^{-2}\right)$, where $r=\sqrt[4]{12}$.  By \eqref{final F(q) evaluation}, the companion series equals a linear combination of $S(1,0,36;2)$, $S(4,0,9;2)$ and elementary constants.  We eliminate $S(4,0,9;2)$ with a result from \cite{ZM}:
\begin{equation}\label{S(4,0,9) eliminator}
\begin{split}
S(1,0,36;t)+S(4,0,9;t)=&\left(1-2^{-t}+2^{1-2t}\right)\left(1+3^{1-2t}\right)L_1(t)L_{-4}(t)\\
&+\left(1+2^{-t}+2^{1-2t}\right)L_{-3}(t)L_{12}(t).
\end{split}
\end{equation}
After noting that $L_1(2)=\frac{\pi^2}{6}$ and $L_{12}(2)=\frac{\pi^2}{6\sqrt{3}}$, we obtain a divergent formula:
\begin{equation*}
\begin{split}
\frac{2}{\pi^2}S(1,0,36;2)=&\frac{49}{18^2}L_{-4}(2)
+ \frac{11}{48 \sqrt3}L_{-3}(2)\\
&-\left(\frac{161+93\sqrt{3}}{18\sqrt[4]{12}}\right)\Re\left[
\hpg54{1,\,1,\,1,\,1,\,\frac{21+\sqrt{3}}{12}}{\frac32,\frac32,\frac32,\frac{9+\sqrt{3}}{12}}{\,(7+4\sqrt{3})^2}
\right].
\end{split}
\end{equation*}
Many additional divergent formulas exist.  We consider these formulas disappointing, because they appear to be quite useless.  Rapidly converging formulas are more exciting, but trickier to produce.

Consider the restriction on $q$ imposed in Proposition \ref{Proposition:Companion series reduction}.  To obtain an $s=\frac12$ companion series from \eqref{5F4 1/2 case in terms of F}, we must select $q$ to lie in a neighborhood of zero.  Unwinding the proof of Theorem \ref{thm extended coeff}, shows that we can only select values of $q$ for which
\begin{equation*}
\theta_3^4(q)=
\hpg32{\frac12,\,\frac12,\,\frac12}{1,1}{\,4\frac{\theta_3^4(-q)}{\theta_3^4(q)}\left(1-\frac{\theta_3^4(-q)}{\theta_3^4(q)}\right)}
\end{equation*}
holds (similar restriction exist when $s=\frac13$ and $s=\frac14$).  This constraint implies that the allowable values on the real axis are $q\in(-1,e^{-\pi})$. If $q\in(-e^{-\pi\sqrt{2}},e^{-\pi})$ then $|z|<1$, and the companion series diverges.  On the other hand, if $q\in(-1,-e^{-\pi\sqrt{2}})$ then $|z|>1$, and we obtain convergent formulas.  Suppose that $q=e^{2\pi i(\frac12+i y)}$, so that $q$ lives on the negative real axis.  Then by \eqref{final F(q) evaluation} we find
\begin{align}\label{Fq and Fq4 in the final section}
\begin{split}
F(q)=&F\left(-e^{-2\pi y}\right)=\frac{120y^3}{\pi^2}S\left(1,1,\frac{1}{4}+y^2;2\right),\\
F(q^4)=&F\left(e^{-8\pi y}\right)=\frac{120(4y)^3}{\pi^2}S\left(1,0,16y^2;2\right).
\end{split}
\end{align}
Trivial manipulations suffice to prove
\begin{equation}\label{S linear combo}
S\left(1,1,\frac{1}{4}+y^2;t\right)=-S(1,0,y^2;t)+18S(1,0,4y^2;t)-16S(1,0,16y^2;t).
\end{equation}

Now we prove the formula for $S(1,0,36;2)$ quoted in the introduction (equation \eqref{S(1,0,36) intro formula}).  Set $q=-e^{-\pi/3}$ in \eqref{5F4 1/2 case in terms of F}.  Using the results above (with $y=\frac16$), we conclude
\begin{align*}
F\left(-e^{-\pi/3}\right)=&\frac{90}{\pi^2}\left(9 S(1, 0, 9;2) - 8 S(1, 0, 36;2) -
 8 S(4, 0, 9;2)\right)\\
F\left(e^{-4\pi/3}\right)=&\frac{2880}{\pi^2}S(4,0,9;2).
\end{align*}
We can eliminate $S(4,0,9;2)$ with \eqref{S(4,0,9) eliminator}, and $S(1,0,9;2)$ disappears using
\begin{equation*}
S(1,0,9;t)=(1+3^{1-2t})L_{1}(t)L_{-4}(t)+L_{-3}(t)L_{12}(t).
\end{equation*}
Putting everything together in \eqref{5F4 1/2 case in terms of F}, and simplifying $(a,b,z)$ with \eqref{param s=1/2}, produces the desired formula for $S(1,0,36;2)$.

Next consider \eqref{5F4 1/2 case in terms of F} when $q=-e^{-\pi/\sqrt{5}}$.  Applying \eqref{Fq and Fq4 in the final section} and \eqref{S linear combo} with $y=\frac{1}{\sqrt{20}}$, reduces the formula to a linear combination of $S(1,0,20;2)$, $S(4,0,5;2)$ and $S(1,0,5;2)$.  We can eliminate the latter two sums with
\begin{align*}
S(4,0,5;t)+S(1,0,20;t)=&(1-2^{-t} + 2^{1-2t} ) L_{1}(t)L_{-20}(t) + (1 + 2^{-t} + 2^{1-2t} ) L_{-4}(t) L_{5}(t)\\
S(1,0,5;t)=&L_{1}(t)L_{-20}(t)+L_{-4}(t)L_{5}(t).
\end{align*}
John Zucker provided the first identity, and the second appears in \cite{GZ}.  Thus we arrive at
\begin{equation}\label{S(1,0,20) formula}
\frac{16 \sqrt{5}}{\pi ^2}S(1,0,20;2)=\frac{5\sqrt{5}}{3} L_{-20}(2)+\frac{104}{25}L_{-4}(2)-\sum_{n=1}^{\infty}\frac{(1)_n^3}{\left(\frac12\right)_n^3}\frac{(a-b n)}{n^3} z^{-n}
\end{equation}
where
\begin{align*}
z=&-8 \left(617+276 \sqrt{5}+2 \sqrt{5 \left(38078+17029 \sqrt{5}\right)}\right)\\
a=&\frac{34}{5}+3 \sqrt{5}+\frac{1}{2} \sqrt{\frac{9032}{25}+\frac{808}{\sqrt{5}}}\\
b=&16+7 \sqrt{5}+\frac{1}{2} \sqrt{\frac{9728}{5}+\frac{4352}{\sqrt{5}}}.
\end{align*}
This formula also converges rapidly, because $z\approx -1.9\times 10^{4}$.

We conclude the paper with one final example.  To obtain a formula for $S(1,0,52;2)$, set $q=-e^{-\pi/\sqrt{13}}$ in \eqref{5F4 1/2 case in terms of F}.  Applying \eqref{Fq and Fq4 in the final section} and \eqref{S linear combo} with $y=\frac{1}{\sqrt{52}}$, reduces the companion series to an expression involving $S(1,0,52;2)$, $S(4,0,13;2)$, and $S(1,0,13;2)$.  The latter two sums can be eliminated with
\begin{align*}
S(1,0,52;t)+S(4,0,13;t)=&(1-2^{-t} + 2^{1-2t} ) L_{1}(t)L_{-52}(t) + (1 + 2^{-t} + 2^{1-2t} ) L_{-4}(t) L_{13}(t)\\
S(1,0,13;t)=&L_{1}(t)L_{-52}(t)+L_{-4}(t)L_{13}(t).
\end{align*}
Zucker provided the first formula, and the second appears in \cite{GZ}.  Therefore, we obtain
\begin{equation}
\frac{16 \sqrt{13}}{\pi^2} S(1,0,52;2)=\frac{5\sqrt{13}}{3}L_{-52}(2)+8 L_{-4}(2)-\sum_{n=1}^{\infty}\frac{(1)_n^3}{\left(\frac12\right)_n^3}\frac{(a-b n)}{n^3} z^{-n},
\end{equation}
where
\begin{align*}
z=&-8 \left(3367657+934020 \sqrt{13}+90 \sqrt{2800274982+776656541 \sqrt{13}}\right),\\
a=&\frac{4266}{13}+91 \sqrt{13}+\frac{1}{13} \sqrt{2 \left(18194697+5046301 \sqrt{13}\right)},\\
b=&720+\frac{2595}{\sqrt{13}}+\frac{48}{26} \sqrt{13 \left(23382+6485 \sqrt{13}\right)}.
\end{align*}
Notice that $z\approx-1.07\times10^8$, so the formula converges rapidly.

\section{Conclusion}
In conclusion, it might be interesting to try to classify all of the values of $S(A,B,C;2)$ which can be treated using the ideas in Section \ref{sec: irreducible values of S}.  It would also be extremely interesting if the methods from Section \ref{sec:completing the hypergeometric} could be used to say something about $3$-dimensional lattice sums such as the Madelung constant.

\begin{acknowledgements}
The authors thank Ross McPhedran and John Zucker for the kind comments and useful suggestions.  The authors are also grateful to
Zucker for providing the
evaluations of $S(1,0,52)+S(4,0,13)$ and $S(1,0,20)+S(4,0,5)$.
\end{acknowledgements}

\end{document}